\pdfoutput=1
\RequirePackage{ifpdf}
\ifpdf 
\documentclass[pdftex]{sigma}
\else
\documentclass{sigma}
\fi

\usepackage{mathtools}
\usepackage[dvipsnames,svgnames]{xcolor}

\numberwithin{equation}{section}

\usepackage{tikz}
\usetikzlibrary{shapes,arrows,positioning,decorations.markings,decorations.pathreplacing}

\usepackage{array}
\usepackage{adjustbox}
\usepackage{enumerate}

\newcommand{\ssp}{\hspace{1pt}}

\newtheorem{Theorem}{Theorem}[section]
\newtheorem*{Theorem*}{Theorem}
\newtheorem{Corollary}[Theorem]{Corollary}
\newtheorem{Lemma}[Theorem]{Lemma}
\newtheorem{Proposition}[Theorem]{Proposition}
\newtheorem{Conjecture}[Theorem]{Conjecture}

\theoremstyle{definition}
\newtheorem{Definition}[Theorem]{Definition}
\newtheorem{Remark}[Theorem]{Remark}

\begin{document}
\allowdisplaybreaks

\newcommand{\arXivNumber}{2602.04390}

\renewcommand{\PaperNumber}{081}

\FirstPageHeading

\ShortArticleName{Colored Interlacing Triangles and Genocchi Medians}

\ArticleName{Colored Interlacing Triangles and Genocchi Medians}

\Author{Natasha BLITVI\'C~$^{\rm a}$ and Leonid PETROV~$^{\rm b}$}

\AuthorNameForHeading{N.~Blitvi\'c and L.~Petrov}

\Address{$^{\rm a)}$~School of Mathematical Sciences, Queen Mary University of London, UK}
\EmailD{\mail{n.blitvic@qmul.ac.uk}}

\Address{$^{\rm b)}$~Department of Mathematics, University of Virginia, Charlottesville, VA, USA}
\EmailD{\mail{lenia.petrov@gmail.com}}

\ArticleDates{Received February 21, 2026, in final form August 09, 2026; Published online August 25, 2026}

\Abstract{Colored interlacing triangles, introduced by Aggarwal--Borodin--Wheeler (2024), provide the combinatorial framework for the central limit theorem for probability measures arising from the Lascoux--Leclerc--Thibon (LLT) polynomials. Colored interlacing triangles depend on two key parameters: the number of colors $n$ and the depth of the triangle $N$. Recent work of Gaetz--Gao (2025) connects these objects to Schubert calculus and resolves the enumeration for $n=3$ and arbitrary depth $N$. However, the enumerative behavior for general $n$ has remained open. In this paper, we analyze the complementary regime: fixed depth $N=2$ and arbitrary number of colors $n$. We prove that in this setting, colored interlacing triangles considered up to a simultaneous relabeling of the colors are in bijection with Dumont derangements, identifying their enumeration with the Genocchi medians. This connects the probabilistic model to a rich hierarchy of classical combinatorial objects. In particular, under this bijection, the natural decomposition of a triangle into consecutive indecomposable blocks corresponds to the direct-sum decomposition of the associated permutation (readily yielding a generating function identity). Using a~recent bijection of Poddar, Sawant, and Shankar (2026), we further identify these `indecomposable' triangles with pure D-permutations. Furthermore, we introduce a $q$-deformation of this enumeration arising naturally from the LLT transition energy. For identity bottom row, we show that the corresponding statistic is the sum of the lower-crossing and lower-nesting statistics on Dumont derangements, and is additive over indecomposable triangles. This yields a new $q$-analogue of the median Genocchi numbers with an elegant continued-fraction expansion, as well as an interpretation of the indecomposable enumerators as Boolean cumulants of a~family of probability measures on~$\mathbb R$. For arbitrary bottom rows, the failure of additivity is governed by an explicit inversion cross-term, leading to a natural $q$-analogue of the generating function identity for indecomposable triangles. Finally, we present computational results and sampling algorithms for colored interlacing triangles with higher $N$ or $n$, which suggests the limits of combinatorial tractability in the $(N,n)$ parameter space, outside of the $N=2$ or $n=3$ regimes.}


\Classification{05A15; 60C05; 05E05; 05A19; 11B68}

\section{Introduction}
\label{sec:intro}

\subsection{Overview}

Colored interlacing triangles were introduced by Aggarwal--Borodin--Wheeler
\cite{ABW2023coloured_LLT} as discrete objects arising in the study of Gaussian
limits of probability measures coming from Cauchy identities for LLT
(Lascoux--Leclerc--Thibon) polynomials.
The pre-limit \emph{LLT processes} live on tuples of semistandard Young
tableaux, and generalize Schur processes of Okounkov--Reshetikhin
\cite{okounkov2003correlation}.
The central limit theorem behavior of the latter is described by the joint law
of eigenvalues of corners of random matrices from the Gaussian unitary ensemble
(GUE), see Okounkov--Reshetikhin~\cite{OkounkovReshetikhin2006RandomMatr} and
Gorin--Panova~\cite{GorinPanova2012_full}.
In the LLT generalization (which also contains a deformation parameter $q$),
along with the continuous GUE corners components, one also obtains a $q$-measure
on the space of certain discrete objects, the \emph{colored interlacing
triangles}.

In this paper, we study combinatorial aspects of colored interlacing triangles
in their own right, focusing on enumeration, bijections, $q$-counting, and
random sampling.
In particular, we obtain an exact enumeration of depth-$2$ triangles in terms of
Genocchi medians together with the finer structure of direct-sum
decompositions, explore a $q$-deformation of the general enumeration problem,
and provide computational results for higher depths, including a disproof of
\cite[Conjecture~A.5]{ABW2023coloured_LLT}.

\subsection{Colored interlacing triangles}

Fix $n,N\ge1$.
	We call $N$ the \emph{depth} of a triangle, and $n$ the \emph{number of
	colors}.
	A~(\emph{colored}) \emph{$n$\nobreakdash-triangle} $\boldsymbol\lambda$ is a
	collection of $nN(N+1)/2$ colored dots (where colors are represented by numbers), arranged
	into $n$ triangles of depth $N$, with each triangle consisting of $1$ dot at
	the bottommost level, $2$ dots at the next level, and so on, up to $N$ dots
	at the top level.

	We denote by $\lambda[i]^{k}_{j} \in\left\{ 1,\dots,n  \right\}$ the color
	of the dot which is located in the $i$-th triangle, on the $k$-th level, and
	at the $j$-th position from the left.
	For any level $k=1,\dots,N-1$, define a~\emph{linear order} between dots
	of all triangles at
	levels $k$ and $k+1$ as follows:
	\begin{align}
		& \lambda[1]^{k+1}_{1} \big\vert
		\lambda[1]^{k}_{1} \big\vert \lambda[1]^{k+1}_{2} \big\vert \cdots \big\vert
		\lambda[1]^{k+1}_{k} \big\vert \lambda[1]^k_k \big\vert
		\lambda[1]^{k+1}_{k+1}
		\big\|
		\lambda[2]^{k+1}_1 \big\vert \lambda[2]^k_1 \big\vert
		\cdots \nonumber\\
&\qquad \big\vert \lambda[n]^{k+1}_k \big\vert \lambda[n]^k_k \big\vert
		\lambda[n]^{k+1}_{k+1},\label{eq:interlacing_order}
	\end{align}
	where the vertical bars separate the entries from different levels,
	and ``$\|$'' separates the first triangle from the second one.
	Denote by $\lambda^{k}=(\lambda^k_1,\lambda^k_2,\dots,\lambda^k_{nk})$ the
	sequence of elements in the $k$-th row, read from left to right according to
	the linear order \eqref{eq:interlacing_order}.
	This notation suppresses the triangle index $[i]$, treating the entire level
	as a single row of length $nk$, which is especially helpful when dealing
	with small-$N$ cases.

\begin{Definition}
	\label{def:colored_interlacing_triangle} The $n$-triangle $\boldsymbol\lambda$ is \emph{interlacing} if it additionally satisfies two conditions:
	\begin{itemize}\itemsep=0pt
		\item For each color $b\in \{ 1,\dots,n  \}$
		and any level number $k=1,\dots,N$, there are exactly $k$ dots of color
		$b$ on the $k$-th level in the union of all triangles.
		\item
		For each color $b\in \{ 1,\dots,n  \}$ and any level
		$k=1,\dots,N-1$, the configurations of color-$b$ dots at levels $k$ and
		$k+1$ must \emph{interlace} (notation $\lambda^k\prec \lambda^{k+1}$).
		That is, between any two consecutive dots of color $b$ at level $k+1$,
		there is exactly one dot of color $b$ at level $k$.
		Here ``between'' is understood in the sense of the linear order~\eqref{eq:interlacing_order}.
	\end{itemize}
	See Figure~\ref{fig:small_interlacing_example} for an example.
\end{Definition}

	We denote the set of all colored interlacing $n$-triangles of depth $N$,
	\[
		\boldsymbol\lambda=\bigl\{\lambda[i]^{k}_{j}\mid 1\le i\le n,\, 1\le
		j\le k\le N\bigr\},
	\]
	by
	$\mathcal{T}_N(n)$.
	Let $T_N(n)\coloneqq\#\mathcal{T}_N(n)$ be the number of such triangles.

For future reference, let us denote by $\mathring{\mathcal{T}}_N(n)$ the subset
of $\mathcal{T}_N(n)$ consisting of triangles with the identity permutation of
colors at the bottom row, $\lambda^1 = (1,2,\dots,n)$.
Clearly, simultaneous permutations of all $n$ colors do not affect the
interlacing condition, so we have
\begin{equation}
	\label{eq:fixed_bottom_row} T_N(n) = n!
                                        \cdot \# \mathring{\mathcal{T}}_N(n).
\end{equation}

\begin{figure}[t]
	\centering
	\begin{tikzpicture}
		[scale=0.75, dot1/.style={circle, fill=red!80, minimum size=5.5mm, inner
		sep=0pt, font=\small\bfseries}, dot2/.style={circle,
		fill=green!70!black, minimum size=5.5mm, inner sep=0pt,
		font=\small\bfseries, text=white}, dot3/.style={circle, fill=blue!70,
		minimum size=5.5mm, inner sep=0pt, font=\small\bfseries, text=white}]
		\def\h{0.8} \def\v{1.05} \def\tsep{5.2*\h}
		\node[left] at (-4*\h, 0) {$k=1$}; \node[left] at (-4*\h, \v) {$k=2$};
		\node[left] at (-4*\h, 2*\v) {$k=3$};
		\draw[gray, densely dashed] (-3.5*\h,0) -- (14*\h,0); \draw[gray,
		densely dashed] (-3.5*\h,\v) -- (14*\h,\v); \draw[gray, densely dashed]
		(-3.5*\h,2*\v) -- (14*\h,2*\v);
		\begin{scope}[shift={(0,0)}]
			\draw[gray!40, thick] (-2*\h,2*\v) -- (2*\h,2*\v) -- (0,0) -- cycle;
			\draw[gray!40, thick] (-\h,\v) -- (\h,\v); \node[dot2] at
			(-2*\h,2*\v) {2}; \node[dot1] at (0,2*\v) {1}; \node[dot2] at
			(2*\h,2*\v) {2}; \node[dot2] at (-\h,\v) {2}; \node[dot1] at (\h,\v)
			{1}; \node[dot2] at (0,0) {2}; \node at (0,2*\v+0.7) {$i=1$};
		\end{scope}
		\begin{scope}[shift={(\tsep,0)}]
			\draw[gray!40, thick] (-2*\h,2*\v) -- (2*\h,2*\v) -- (0,0) -- cycle;
			\draw[gray!40, thick] (-\h,\v) -- (\h,\v); \node[dot3] at
			(-2*\h,2*\v) {3}; \node[dot3] at (0,2*\v) {3}; \node[dot1] at
			(2*\h,2*\v) {1}; \node[dot3] at (-\h,\v) {3}; \node[dot3] at (\h,\v)
			{3}; \node[dot3] at (0,0) {3}; \node at (0,2*\v+0.7) {$i=2$};
		\end{scope}
		\begin{scope}[shift={(2*\tsep,0)}]
			\draw[gray!40, thick] (-2*\h,2*\v) -- (2*\h,2*\v) -- (0,0) -- cycle;
			\draw[gray!40, thick] (-\h,\v) -- (\h,\v); \node[dot3] at
			(-2*\h,2*\v) {3}; \node[dot2] at (0,2*\v) {2}; \node[dot1] at
			(2*\h,2*\v) {1}; \node[dot2] at (-\h,\v) {2}; \node[dot1] at (\h,\v)
			{1}; \node[dot1] at (0,0) {1}; \node at (0,2*\v+0.7) {$i=3$};
		\end{scope}
		\begin{scope}[shift={(3.3*\tsep,0)}, scale=0.9]
			\draw[gray!40, thick] (-2*\h,2*\v) -- (2*\h,2*\v) -- (0,0) -- cycle;
			\draw[gray!40, thick] (-\h,\v) -- (\h,\v); \node[font=\footnotesize]
			at (-2*\h,2*\v) {$\lambda[i]^3_1$}; \node[font=\footnotesize] at
			(0,2*\v) {$\lambda[i]^3_2$}; \node[font=\footnotesize] at
			(2*\h,2*\v) {$\lambda[i]^3_3$}; \node[font=\footnotesize] at
			(-\h,\v) {$\lambda[i]^2_1$}; \node[font=\footnotesize] at (\h,\v)
			{$\lambda[i]^2_2$}; \node[font=\footnotesize] at (0,0)
			{$\lambda[i]^1_1$}; \node at (0,2*\v+0.8) {\small indexing};
		\end{scope}
	\end{tikzpicture}\vspace{-1.0mm}

	\caption{Left: A colored interlacing triangle with $n=3$ colors and $N=3$
	levels.
	Colors: \textcolor{red!80}{1 (red)}, \textcolor{green!70!black}{2 (green)},
	\textcolor{blue!70}{3 (blue)}.
	Right: indexing notation $\lambda[i]^k_j$ for the $i$-th triangle.
	Row sequences:
	$\lambda^{1}=(\textcolor{green!70!black}{2},\textcolor{blue!70}{3},\textcolor{red!80}{1})$, $\lambda^{2}=(\textcolor{green!70!black}{2},\textcolor{red!80}{1},\textcolor{blue!70}{3},\textcolor{blue!70}{3},\textcolor{green!70!black}{2},\textcolor{red!80}{1})$, $\lambda^{3}=(\textcolor{green!70!black}{2},\textcolor{red!80}{1},\textcolor{green!70!black}{2},\textcolor{blue!70}{3},\textcolor{blue!70}{3},\textcolor{red!80}{1},\textcolor{blue!70}{3},\textcolor{green!70!black}{2},\textcolor{red!80}{1})$.
	Each pair of consecutive rows may be written in one line using the linear
	order \eqref{eq:interlacing_order}, for example, $(2|2|1\|3|3|3\|2|1|1)$ for the
	first two levels.} \label{fig:small_interlacing_example}\vspace{-2.0mm}
\end{figure}

Let us briefly mention the probabilistic origin of these
objects from Aggarwal--Borodin--Wheeler
\cite{ABW2023coloured_LLT}.  The LLT
(Lascoux--Leclerc--Thibon) polynomials
\cite{lascoux1997ribbon} are a~$q$-deformation
of products of $n$ Schur polynomials,
reducing to the straightforward product
when $q=1$ (here $n$ is our number of colors).
Similarly to the Schur case
\cite{fulman1997probabilistic,okounkov2001infinite,okounkov2003correlation}, one can define
probability measures by normalizing Cauchy summation
identities.  Among the measures one can obtain in this
manner are coupled tuples of domino tilings of the Aztec
diamond introduced by Corteel--Gitlin--Keating~\cite{corteel2023ktilings}, see also
\cite{keating2024shuffling} for a sampling algorithm
generalizing classical domino shuffling to the coupled
setting.

Aggarwal--Borodin--Wheeler \cite{ABW2023coloured_LLT} established a central limit
theorem for measures based on LLT polynomials.
In the Schur case, the role of the Gaussian distribution in this regime is
played by the corners process of the Gaussian unitary ensemble (GUE) of random
matrix theory~\cite{GorinPanova2012_full,OkounkovReshetikhin2006RandomMatr}.
That is, one considers the joint distribution of eigenvalues of all top-left
corners of a random Gaussian Hermitian matrix, and this joint distribution
describes the limiting fluctuations of Schur processes.

Generalizing to the LLT level, one obtains a more complex limit object.
Namely, the distribution describing the fluctuations of LLT processes splits
into two independent parts: a \emph{continuous part}
consisting of $n$
independent GUE (Gaussian unitary ensemble) corners processes, and a~\emph{discrete part}, which is a $q$-weighted probability distribution on the
finite set $\mathcal{T}_N(n)$ of colored interlacing triangles from
Definition~\ref{def:colored_interlacing_triangle}.
The discrete part is nontrivial for $n>1$, while for $n=1$ (Schur case), one simply has
a single GUE corners process.

The combinatorial landscape of colored interlacing triangles is vast,
parameterized by the number of colors $n$ and the depth $N$.
A complete understanding requires exploring both dimensions.
Aggarwal--Borodin--Wheeler \cite{ABW2023coloured_LLT} initiated the study of
enumeration for fixed $n$, observing that $T_N(2) = 2^N$ since two-colored
interlacing triangles are isomorphic to binary trees
\cite[Proposition~A.1]{ABW2023coloured_LLT}.
They conjectured \cite[Conjecture~A.3]{ABW2023coloured_LLT} that
$T_N(3) = \frac{1}{4}\mathfrak{g}_N^{\Delta}(4)$, where
$\mathfrak{g}_N^{\Delta}(4)$ \cite[A153467]{OEIS} is the number of $4$-colorings
of a triangular grid of side length $N$.
This was proved by Gaetz and Gao \cite{gaetz2023bijection}, who uncovered deep
connections to Schubert calculus.
We independently observed that \cite[Conjecture~A.5]{ABW2023coloured_LLT}, which
posited $T_N(4) = \frac{1}{5}\mathfrak{g}_N^{*}(5)$ (where
$\mathfrak{g}_N^{*}(5)$ \cite[A068294]{OEIS} counts $5$-colorings of an
octagonal array), fails for $N \ge 3$; this was also reported by
\cite{gaetz2023bijection}.

The present work complements these fixed-$n$ results by exploring the
``horizontal'' direction ($N=2$, arbitrary $n$).
By fixing the depth, we let the number of colors grow indefinitely, revealing a
connection to the classical Genocchi numbers that is invisible when $n$ is
fixed.
This result, together with \cite{gaetz2023bijection}, maps out the landscape of
so far available enumerative results for colored interlacing triangles: the
vertical direction connects to geometry, while the horizontal direction connects
to the classical theory of permutation statistics, and presents a
straightforward path to interpreting $q$-analogs coming from the central limit
theorem for measures based on LLT polynomials.\looseness=-1

In the next Section~\ref{subsec:main_results}, we summarize our main results on
enumeration and $q$-counting in the discrete space $\mathcal{T}_N(n)$.

\subsection{Main results and outline of the paper}
\label{subsec:main_results}

In Section~\ref{sec:genocchi_case}, we prove that depth-$2$ colored interlacing triangles considered up to a simultaneous relabeling of the colors are in bijection with Dumont derangements (Section~\ref{subsec:dumont_bijection}), which are counted by the Genocchi medians $H_n$.  This connects the probabilistic model to classical combinatorial objects.  The bijection is simple and explicit, and it carries additional structure. In particular, in Section~\ref{subsec:indecomposable_pure}, we show that the bijection matches the visible splitting of triangles with the direct-sum decomposition on the permutation side. That is, `indecomposable' triangles correspond to indecomposable Dumont derangements.  Drawing on very recent work of \cite{PoddarSawantShankar2026}, we then exhibit a bijection between the indecomposable Dumont derangements and the \emph{pure} $D$-permutations of \cite{DebSokal2024} (Proposition~\ref{prop:pure_indecomposable_bijection}).  Although these two families were known to be equinumerous (as both are counted by the sequence $h^\flat$ of \cite{DebSokal2024}, \cite[A362111]{OEIS}, whose generating function has an elegant Stieltjes continued fraction expansion), this explicit correspondence between them appears not to have been noted previously.

In Section~\ref{sec:q_enumeration}, we analyze the intrinsic LLT $q$-statistic $\psi$ defined in \cite{ABW2023coloured_LLT}, which measures the ``energy'' of transitions between consecutive levels of a triangle.  We show that for $N=2$, weighting the triangles by $q^{\psi}$ leads to $q$-analogs of the Genocchi medians different from the standard ones (we recall the latter in Section~\ref{subsec:known_q_analogs}).  Moreover, $\psi$ admits a combinatorial description which respects the interval decomposition of Section~\ref{sec:genocchi_case}, in the sense that the weight $q^{\psi}$ is multiplicative over the indecomposable blocks (Proposition~\ref{prop:psi_multiplicative}).  In particular, this produces a~polynomial $q$-analog of the pure $D$-permutation counts $h^\flat$ (Remark~\ref{rem:indecomposable_triangles}).  For arbitrary bottom rows, the additivity acquires an explicit inversion cross-term, leading to a $q$-exponential form of the generating function (Theorem~\ref{thm:q_exponential_first_return}).  Over the identity bottom row, $\psi$ moreover coincides with the number of lower crossings and nestings of the associated Dumont derangement (Proposition~\ref{prop:psi_crossings_nestings}).  This yields a Stieltjes-type continued fraction for the resulting $q$-analog and, in Section~\ref{subsec:moments_cumulants}, identifies the indecomposable enumerators $\mathrm{Irr}_n(q)$ as Boolean cumulants.

Probing the hierarchy beyond $N=2$, in Section~\ref{subsec:computational_approach} we enumerate $\mathcal{T}_N(n)$ for a range of depths and sizes, finding that the counts exhibit large prime factors and eventually hit computational walls --- evidence that the combinatorics diverges from the closed-form patterns of the Genocchi case at $N=2$, and from the Schubert calculus connections at $n=3$.  Along the way, we provide an independent disproof of \cite[Conjecture~A.5]{ABW2023coloured_LLT}.

Returning to depth $N=2$, in Sections~\ref{subsec:q_data}, \ref{subsec:conjectures_on_coefficients} and \ref{subsec:linear_cumulant_phenomenon} we perform a detailed computational study of the $q$-Genocchi medians.  We conjecture formulas for the coefficients of low powers of $q$, and prove the one for the linear coefficient.  For a fixed power of $q$ and growing $n$, we establish a lower bound of the conjectured order, using the additivity of $\psi$ over the interval decomposition.  We also observe a curious ``linear cumulant'' phenomenon in the distribution of $\psi$.

Finally, in Section~\ref{sec:sampling}, we develop a Markov chain Monte Carlo algorithm
for sampling from the $q$-weighted measure on $\mathcal{T}_2(n)$, and use it to
experimentally probe the model for large $n$.
We observe some resemblance to Mallows permutons, and conjecture that the
statistic $\psi$ is asymptotically Gaussian.

\section{Two-level colored interlacing triangles and Genocchi medians}
\label{sec:genocchi_case}

In this section, we establish the enumeration formula for depth-$2$ colored
interlacing triangles in terms of Genocchi medians
(Section~\ref{subsec:dumont_bijection}), derive divisibility properties of the
counts (Section~\ref{subsec:divisibility}), and describe the decomposition
structure transported by the bijection
(Section~\ref{subsec:indecomposable_pure}).

\subsection{The bijection with Dumont derangements}
\label{subsec:dumont_bijection}

The sequence of Genocchi medians $H_n$ (also sometimes called the Genocchi
numbers of second kind) is a classical integer sequence in combinatorics with
several interpretations.
We refer to~\mbox{\cite{Dumont1974,DumontRandrianarivony1994}}, and the entry
\cite[A005439]{OEIS} for background and further references.
The following definition in terms of a certain class of permutations (called the
\emph{Dumont derangements}) goes back to~\cite{Dumont1974}.

\begin{Definition}\label{def:derrangement} A permutation $\sigma\in S_{2n}$ is called a
	\emph{Dumont derangement} (of second kind) if the following two conditions
	hold:
	\begin{equation*}
		\begin{cases}
			i< \sigma(i), & \textnormal{for all $i$ odd};			\\
			i> \sigma(i), & \textnormal{for all $i$ even}.
		\end{cases}
	\end{equation*}
	The number of Dumont derangements in $S_{2n}$ is called the $n$-th
	\emph{Genocchi median} and denoted by~$H_n$.
	We have
	\begin{equation*}
		H_0=H_1=1,\quad H_2=2,\quad H_3=8,\quad H_4=56,\quad H_5=608,\quad
		\dots
	\end{equation*}
\end{Definition}

Recent work often uses the broader language of \emph{$D$-permutations}: these
are permutations $\sigma\in S_{2n}$ satisfying $2i-1\le\sigma(2i-1)$ and
$2i\ge\sigma(2i)$ for all $i$.
The Dumont derangements defined above are precisely the derangements in this
class, also called \emph{$D$-derangements} in this literature. (See
\cite{LazarWachs2019} for this terminology and \cite{DebSokal2024} for
continued-fraction refinements of the Genocchi and median Genocchi numbers.)

Recall that we denote by $T_N(n)$ the number of colored interlacing
$n$-triangles of depth $N$ (Definition~\ref{def:colored_interlacing_triangle}).

\begin{Theorem}
	\label{thm:Genocchi_counts} We have $T_2(n)=n!\cdot H_n$ for all $n\geq 1$.
\end{Theorem}
\begin{proof}
	It suffices to show that $\# \mathring{\mathcal{T}}_2(n) = H_n$, see
	\eqref{eq:fixed_bottom_row}.
	Let us identify the valid top rows with Dumont derangements in~$S_{2n}$.

	The top row is a sequence of $2n$ elements where each color appears exactly
	twice.
	For a~color $c=1,\dots,n $, let $p_1 < p_2$ be the positions of its two
	occurrences, and define $\sigma \in S_{2n}$ by
	\begin{equation*}
		\sigma(2c) = p_1, \qquad \sigma(2c-1) = p_2.
	\end{equation*}
	Since the bottom row is the identity, the level-$1$ dot of color $c$
	belongs to the $c$-th triangle, and in the linear order
	\eqref{eq:interlacing_order} it sits between the level-$2$ positions
	$2c-1$ and $2c$.
	The interlacing condition requires that this dot lie strictly between the
	two level-$2$ occurrences of $c$, that is, $p_1 \le 2c-1$ and
	$p_2 \ge 2c$.
	These are precisely the Dumont derangement conditions $\sigma(2c) < 2c$
	and $\sigma(2c-1) > 2c-1$ of Definition~\ref{def:derrangement}; see
	Figure~\ref{fig:bijection_dumont} for an example.
\end{proof}

\begin{figure}[htbp]
	\centering
	\begin{tikzpicture}[
		scale=0.9,
		dot1/.style={circle, fill=red!80, minimum size=5mm, inner sep=0pt, font=\footnotesize\bfseries},
		dot2/.style={circle, fill=green!70!black, minimum size=5mm, inner sep=0pt, font=\footnotesize\bfseries, text=white},
		dot3/.style={circle, fill=blue!70, minimum size=5mm, inner sep=0pt, font=\footnotesize\bfseries, text=white}]
		\node[font=\footnotesize] at (-0.3,1.3) {$\lambda^2$};
		\node[font=\footnotesize] at (-0.3,0) {$\lambda^1$};
		\node[dot1] (t1) at (1,1.3) {1};
		\node[dot2] (t2) at (3,1.3) {2};
		\node[dot1] (t3) at (4,1.3) {1};
		\node[dot3] (t4) at (6,1.3) {3};
		\node[dot2] (t5) at (7,1.3) {2};
		\node[dot3] (t6) at (9,1.3) {3};
		\node[dot1] (b1) at (2,0) {1};
		\node[dot2] (b2) at (5,0) {2};
		\node[dot3] (b3) at (8,0) {3};
		\draw[red!80, line width=1pt] (t1) to[bend left=45] (t3);
		\draw[green!70!black, line width=1pt] (t2) to[bend left=58] (t5);
		\draw[blue!70, line width=1pt] (t4) to[bend left=45] (t6);
		\foreach \n/\p in {t1/1,t2/2,t3/3,t4/4,t5/5,t6/6}
			\node[below=2pt of \n, font=\scriptsize] {$\p$};
	\end{tikzpicture}

	\smallskip

	\begin{tikzpicture}[scale=0.9, every node/.style={font=\footnotesize}]
		\node[anchor=east] at (0.4,0.55) {$i=$};
		\node[anchor=east] at (0.4,-0.2) {$\sigma(i)=$};
		\foreach \i in {1,...,6} \node at (\i,0.55) {$\i$};
		\foreach \i/\s/\col in {1/3/red!80,2/1/red!80,3/5/green!70!black,4/2/green!70!black,5/6/blue!70,6/4/blue!70}
			\node[text=\col, font=\footnotesize\bfseries] at (\i,-0.2) {$\s$};
	\end{tikzpicture}
	\caption{The bijection of Theorem~\ref{thm:Genocchi_counts} for $n=3$, with identity
	bottom row $\lambda^1=(1,2,3)$.
	Top: a~depth-$2$ colored interlacing triangle in linear order. The level-$1$
	(identity) dots sit between the level-$2$ dots, and the two level-$2$ dots of
	each color interlace around its level-$1$ dot (arcs).
	Number the level-$2$ dots $1,\dots,2n$ from the left; if color $c$ occupies
	positions $p_1<p_2$, set $\sigma(2c-1)=p_2$ and $\sigma(2c)=p_1$.
	Bottom: the resulting permutation $\sigma=(3,1,5,2,6,4)\in S_6$, with entries
	colored by source color.}
	\label{fig:bijection_dumont}
\end{figure}

\subsection{Divisibility}
\label{subsec:divisibility}

The interpretation in terms of colored interlacing triangles
  yields the following divisibility property.

\begin{Proposition}\label{Prop2.3}
For $N\ge2$, the number $T_N(n)/n!$ is
  divisible by $2^{n-1}$.
\end{Proposition}

\begin{proof} Consider the set $\mathring{\mathcal{T}}_N(n)$, whose
  cardinality equals $T_N(n)/n!$ by~\eqref{eq:fixed_bottom_row}. For each
  $i=1,\ldots,n-1$, the entries $\lambda[i]^N_N$ and
  $\lambda[i+1]^N_1$ are immediate neighbors in the linear order~\eqref{eq:interlacing_order}, with no entries between them at level $N-1$. Therefore,
  $\lambda[i]^N_N \ne \lambda[i+1]^N_1$. Swapping these two
  entries produces a different valid top row, and this operation
  can be performed independently for each~$i$. This leads to
  $n-1$ independent involutions on $\mathring{\mathcal{T}}_N(n)$,
  which implies the divisibility.
\end{proof}

  Taking $N=2$ and using Theorem~\ref{thm:Genocchi_counts} recovers the well-known
  divisibility $2^{n-1} \mid H_n$ of the median Genocchi numbers.
  The quotient~$H_n/2^{n-1}$ is often called the \emph{normalized
  median Genocchi number} in the literature.

\begin{Remark}
	\label{rem:divisibility_no_improvement} One might hope to improve the
	divisibility in Proposition~\ref{Prop2.3} to $2^{(n-1)(N-1)}$
	by constructing independent involutions at each of the $N-1$ levels
	$2\le k \le N$.
	However, as the hypothetical involutions at lower levels are constrained by
	the interlacing conditions imposed from above, this may not be possible.
	Indeed, exact enumeration shows that such an improvement is not possible in
	general: $T_3(3)/3! = 88$ is divisible by $2^3$ but not by
	${2^{(3-1)(3-1)} = 16}$ (see Section~\ref{sec:enumeration} and
	Table~\ref{tab:2_adic_valuations} in particular).
\end{Remark}

\subsection[Indecomposable triangles and pure D-permutations]{Indecomposable triangles and pure $\boldsymbol{D}$-permutations}
\label{subsec:indecomposable_pure}

The bijection of Theorem~\ref{thm:Genocchi_counts} reveals
additional structure expressed in terms of interval
decompositions, which we now describe.  Throughout this
subsection, we work with the set
$\mathring{\mathcal{T}}_2(n)$ of depth-$2$ triangles with
identity bottom row.

\begin{Definition}
	\label{def:indecomposable_triangle}
	Let $\boldsymbol\lambda\in\mathring{\mathcal{T}}_2(n)$.
	We say that $\boldsymbol\lambda$ \emph{splits} after the $j$-th triangle
	($1\le j\le n-1$) if no color occurs both among the first $j$ triangles
	and after them; equivalently, no interlacing arc as in
	Figure~\ref{fig:bijection_dumont} crosses the division between the $j$-th and
	the $(j+1)$-st triangles.
	Cutting at every such division decomposes $\boldsymbol\lambda$ uniquely
	into consecutive blocks; shifting the colors identifies each block of
	width $w$ with an element of $\mathring{\mathcal{T}}_2(w)$.
	We call $\boldsymbol\lambda$ \emph{indecomposable} if it does not split
	in this manner.
\end{Definition}

\begin{Proposition}
	\label{prop:splits_direct_sum}
	Under the bijection of Theorem~{\rm \ref{thm:Genocchi_counts}},
	$\boldsymbol\lambda\in\mathring{\mathcal{T}}_2(n)$ splits after the
	$j$-th triangle if and only if the corresponding Dumont derangement
	$\sigma$ maps $\{1,\dots,2j\}$ to itself. Consequently, the indecomposable triangles correspond to the
	\emph{indecomposable} Dumont derangements, those preserving no
	prefix $\{1,\dots,r\}$ with $1\le r\le 2n-1$.
\end{Proposition}
\begin{proof}
	Straightforward from Theorem~\ref{thm:Genocchi_counts}.
\end{proof}

\begin{Corollary}
	\label{cor:indecomposable_count}
	Let $I_n$ denote the number of indecomposable elements of
	$\mathring{\mathcal{T}}_2(n)$.
	Then
	\begin{equation}
		\label{eq:first_return_counting}
		\sum_{n\ge0} H_n\ssp x^n=\biggl(1-\sum_{j\ge1} I_j\ssp
		x^j\biggr)^{-1},
	\end{equation}
	which determines the sequence $I_n=1,1,5,41,493,8161,\dots$.
\end{Corollary}
\begin{proof}
	Since $\#\mathring{\mathcal{T}}_2(n)=H_n$ (proof of
	Theorem~\ref{thm:Genocchi_counts}) and the decomposition of
	Definition~\ref{def:indecomposable_triangle} is unique, recording the width of the
	first block gives the convolution identity
	$H_n=\sum_{j=1}^{n} I_j\ssp H_{n-j}$, $n\ge1$, which is equivalent to
	\eqref{eq:first_return_counting}.
	This is the classical interval decomposition of a
	combinatorial class into indecomposable components. (Compare the
	`indecomposable permutations' of \cite{Comtet1972} and see
	\cite{Beissinger1985} for a more general discussion.)
\end{proof}

\begin{Definition}[\cite{DebSokal2024}]
	\label{def:pure_D_permutations}
	A $D$-permutation is called \emph{pure} if it has no
	record-antirecords, that is, no indices which are simultaneously a
	left-to-right maximum and a right-to-left minimum.
\end{Definition}

Deb and Sokal \cite{DebSokal2024} counted the pure $D$-permutations, finding that
there are exactly $h^\flat_{m+1}$ of these in $S_{2m}$, where $h^\flat$
denotes the sequence
\begin{equation*}
	h^\flat_n=1,1,5,41,493,8161,\dots,
\end{equation*}
see \cite[A362111]{OEIS}.

Denote by $\alpha\oplus\beta$ the direct sum of two permutations, that is, the
concatenation of~$\alpha$ with~$\beta$ shifted up by the size of~$\alpha$.
It was shown in~\cite{DebSokal2024} that a record-antirecord of a~$D$-permutation $\sigma$
amounts to a pair $\{2i-1,2i\}$ of fixed points splitting off the prefix
$\{1,\dots,2i-2\}$. In other words, the two-letter identity permutation~$12$ occurs as a direct summand of~$\sigma$.
Purity thus forbids one specific direct summand, while
indecomposability forbids all splittings. Clearly, the latter is stronger: for
instance, $2143=21\oplus21$ is pure but decomposable, and $S_4$ contains
five pure but only four indecomposable $D$-permutations.

The initial values of $I_n$ in Corollary~\ref{cor:indecomposable_count} match the
sequence $h^\flat_n$.
Recall that
$I_n$ counts the
indecomposable Dumont derangements in $S_{2n}$, while $h^\flat_n$
counts the pure $D$-permutations in~$S_{2n-2}$.
We now identify the two families by an explicit bijection.
For $m\ge1$, Poddar, Sawant, and Shankar
\cite[Theorem~6.2]{PoddarSawantShankar2026} exhibit a simple bijection
$\sigma\mapsto\tau$ from the $D$-permutations in~$S_{2m}$ onto the Dumont
derangements in $S_{2m+2}$, given by
\begin{gather}
\tau(2)=1,\qquad \tau(2m+1)=2m+2,\qquad
	\tau(2i-1)=\sigma(2i-1)+1,\nonumber\\
	\tau(2i+2)=\sigma(2i)+1,\label{eq:PSS_map}
\end{gather}
where $1\le i\le m$.

\begin{Proposition}\label{prop:pure_indecomposable_bijection}
	For $m\ge1$, the bijection \eqref{eq:PSS_map} restricts to a bijection
	from the pure $D$\nobreakdash-per\-mutations in $S_{2m}$ onto the indecomposable
	Dumont derangements in $S_{2m+2}$.
\end{Proposition}
\begin{proof}
	For $1\le k\le m$, the positions $1,\dots,2k$ consist of the odd
	positions $2i-1$ with $1\le i\le k$, the position $2$, and the
	positions $2i+2$ with $1\le i\le k-1$.
	Hence, by \eqref{eq:PSS_map},
	\begin{equation*}
		\tau(\{1,\dots,2k\})=\{1\}\cup
		\bigl(\sigma(\{1,\dots,2k-1\})+1\bigr),
	\end{equation*}
	so $\tau$ preserves the even prefix $\{1,\dots,2k\}$ if and only if
	$\sigma$ preserves the odd prefix $\{1,\dots,{2k-1}\}$.
	By the record-antirecord characterization recalled after
	Definition~\ref{def:pure_D_permutations}, $\sigma$ is pure if and only if it
	preserves no odd prefix; and since a Dumont derangement never preserves
	an odd prefix (preserving $\{1,\dots,2k-1\}$ would force
	$\tau(2k-1)\le 2k-1$, contradicting Definition~\ref{def:derrangement}), $\tau$~is indecomposable if and only if it preserves no even prefix.
	The restriction claim follows.
	Combining it with Proposition~\ref{prop:splits_direct_sum} and with
	$h^\flat_{m+1}=\#\{\text{pure $D$-permutations in }S_{2m}\}$
	\cite{DebSokal2024} yields $I_{m+1}=h^\flat_{m+1}$ for $m\ge1$, while
	$I_1=1=h^\flat_1$ directly.
\end{proof}

As an example, for $m=2$, the map \eqref{eq:PSS_map} sends the five pure
$D$-permutations $2143$, $3142$, $3241$, $4132$, $4231$ to the five
indecomposable Dumont derangements in $S_6$.
Combining Proposition~\ref{prop:pure_indecomposable_bijection} with the continued fraction for pure $D$-permutations from~\cite{DebSokal2024}, we obtain the following Stieltjes continued fraction for indecomposable Dumont derangements (equivalently, our indecomposable colored interlacing triangles).

\begin{Corollary}\label{cor:indecomposable_Sfrac} As formal power series,
\begin{equation*} \sum_{n\ge0}I_{n+1}\ssp x^n=\cfrac{1}{1-\cfrac{\alpha_1 x}{1-\cfrac{\alpha_2 x}{1-\cdots}}},
\end{equation*}
 with coefficients $\alpha_{2k-1}=k^2$
and $\alpha_{2k}=(k+1)^2$.
\end{Corollary}

\begin{proof}
Follows from
the continued fraction for pure $D$-permutations, given in \cite{DebSokal2024}  (see their sequence
$(h^\flat_{n+1})_{n\ge0}$).
\end{proof}

\begin{Remark}\label{rem:classical_Sfrac}
Compare Corollary~\ref{cor:indecomposable_Sfrac} to the classical
continued fraction for the Genocchi medians, given by $\alpha_{2k-1}=\alpha_{2k}=k^2$, further discussed in
Section~\ref{subsec:moments_cumulants} below.\end{Remark}

\begin{Remark}
	\label{rem:hyperplane_regions}
	The bijection \eqref{eq:PSS_map} also places our triangles in a geometric
	context.
	Fang, Fu, Li, and Yuan
	\cite[Theorem~4.8]{FangFuLiYuan2025} construct a bijection between the
	$D$-permutations in $S_{2m}$ and the permutations of $[2m]$ all of whose
	ascents go from an odd letter to an even one.
	By \cite[Theo\-rem~1.4]{FangFuLiYuan2025}, the latter label the regions of
	the hyperplane arrangement
	$\mathcal{K}_{2m}=\{{x_{2i-1}-x_{2j}=0}\mid 1\le i\le j\le m\}$ of
	\cite{LazarWachs2019}, which is linearly isomorphic to the homogenized
	Linial arrangement of~\cite{Hetyei2019}.
	Composing with \eqref{eq:PSS_map}, we obtain an explicit bijection between
	$\mathring{\mathcal{T}}_2(m+1)$ and the regions of $\mathcal{K}_{2m}$.
\end{Remark}

In the following section, we refine this enumeration for the $q$-statistic
introduced in \cite{ABW2023coloured_LLT},
which, as we show, translates into natural combinatorial structure.

\section[The q-enumeration of colored interlacing triangles]{The $\boldsymbol{q}$-enumeration of colored interlacing triangles}
\label{sec:q_enumeration}

In this section, we explore the $q$-counting statistic, which we denote by
$\psi(\lambda^{k},\lambda^{k+1})$, on colored interlacing triangles based on
interactions between consecutive levels.
This statistic
arises from the
central limit theorem for LLT processes \cite[Section~1.6]{ABW2023coloured_LLT} (in that paper,
it is denoted by~$\xi(\cdot,\cdot)$).
The $\psi$-statistic and the $q$-counting polynomial introduced below are
defined for general depth~$N$, and we keep the general $N$ setting here for completeness.
Note, however, that most results in this paper are specific to the depth $N=2$ case.

For depth $N=2$, we obtain new $q$-analogs of the quantities
$T_2(n)=n!\cdot H_n$, different from the standard $q$-analogs of the Genocchi
medians $H_n$ recalled below.
We begin with a brief survey of three known $q$-analogs of the Genocchi medians.

\subsection[Known q-analogs of Genocchi medians]{Known $\boldsymbol{q}$-analogs of Genocchi medians}
\label{subsec:known_q_analogs}

Several $q$-analogs of the Genocchi medians can be obtained by $q$-deforming the
various classical definitions of these numbers, including generating functions,
recurrences, and combinatorial models obtained by assigning a $q$-weight to the
relevant objects (in particular, Dumont derangements).

First, Randrianarivony \cite{Randrianarivony1997} simply weights each Dumont
derangement $\sigma\in S_{2n}$ by the usual inversion number
$\mathrm{inv}(\sigma)$, and also provides Stieltjes-type continued fraction
expansions for the corresponding generating function.
The first few $q$-deformations from \cite{Randrianarivony1997} are
\begin{equation}
	\label{eq:Randrianarivony_q_Gandhi}
	\begin{split}
		&H_1^R(q)= q, \\
		&H_2^R(q)= q^2 (1 + q), \\
		&H_3^R(q)= q^3 (1 + q)^3 \bigl(1 - q + q^2\bigr), \\
		&H_4^R(q)= q^4 (1 + q)^3 \bigl(1 + 2q^3 + q^5 + 2q^6 + q^7\bigr), \\
		&H_5^R(q)= q^5 (1 + q)^5 (1 - q + q^2)^2 \bigl(1 + q + q^3 + q^4 + 2q^5 +
		4q^6 + 5q^7 + 3q^8 + q^9\bigr).
	\end{split}
\end{equation}
Second, Han and Zeng \cite{HanZeng1999,HanZeng1999Discrete}
introduced polynomials $C_n(x,q)$ via a $q$-deformation of the classical Gandhi
recurrence:
\begin{gather*}
	\Delta_q f(x) = \frac{f(1+qx)-f(x)}{1+(q-1)x}, \qquad C_1(x,q)=1, \\
	C_n(x,q) = (1+qx) \Delta_q  ( x C_{n-1}(x,q) ), \qquad n \ge 2.
\end{gather*}
The values at $x=1$ yield the following $q$-analogs of the Genocchi medians:
\begin{equation}
	\label{eq:HanZeng_q_Gandhi}
	\begin{split}
		&H_1^{HZ}(q)= 1, \\
&H_2^{HZ}(q)= 1 + q, \\
&H_3^{HZ}(q)= (1+q)^3, \\
		&H_4^{HZ}(q)= (1+q)^3 \bigl(1 + 3q + 2q^2 + q^3\bigr), \\
		&H_5^{HZ}(q)= (1+q)^5 \bigl(1 + 5q + 5q^2 + 5q^3 + 2q^4 + q^5\bigr).
	\end{split}
\end{equation}
These polynomials have a combinatorial interpretation in terms of a certain
\emph{Denert statistic} on Dumont derangements \cite{HanZeng1999Discrete}, which
separately treats inversions between odd and even positions.
Feigin \cite{Feigin2012} identified the polynomials \eqref{eq:HanZeng_q_Gandhi}
(divided by $(1+q)^{n-1}$) as the Poincar\'e polynomials of degenerate flag
varieties.

Finally, the third $q$-analog was introduced by Zeng and Zhou
\cite{ZengZhou2006} via a $q$-deformation of the Seidel triangle recurrence~\cite{Seidel1877}:
\begin{align*}
	&g_{1,1}(q)= g_{2,1}(q)=1,\\
	&g_{2i+1,j}(q)= g_{2i+1,j-1}(q)+q^{j-1}g_{2i,j}(q), \qquad 1\le j\le i+1, \\
	&g_{2i,j}(q)= g_{2i,j+1}(q)+q^{j-1} g_{2i-1,j}(q), \qquad 1\le j\le i,
\end{align*}
with the boundary conditions $g_{2i+1,0}(q)=0$ and $g_{2i,i+1}(q)=0$ for all
$i\ge1$.
The $q$-Genocchi medians are $H_n^{ZZ}(q) := g_{2n,1}(q)$:
\begin{equation}
	\label{eq:ZengZhou_q_Genocchi}
	\begin{split}
		&H_1^{ZZ}(q)= 1, \\ &H_2^{ZZ}(q)= 1 + q, \\ &H_3^{ZZ}(q)= (1 + q)^2 \bigl(1 + q^2\bigr), \\
		&H_4^{ZZ}(q)= (1 + q)^2 \bigl(1 + q^2\bigr) \bigl(1 + q + q^2 + 2q^3 + q^4 + q^5\bigr), \\
		&H_5^{ZZ}(q)= (1 + q)^3 \bigl(1 + q^2\bigr)^2 \bigl(1 - q + q^2\bigr) \bigl(1 + 2q + 2q^2 + 3q^3
		+ 4q^4 + 4q^5 + 2q^6 + q^7\bigr).
	\end{split}
\end{equation}
These can also be interpreted in terms of a charge statistic on so-called strict
alternating pistols~\cite{ZengZhou2006}.

We see that all of these $q$-analogs \eqref{eq:Randrianarivony_q_Gandhi},
\eqref{eq:HanZeng_q_Gandhi}, and
\eqref{eq:ZengZhou_q_Genocchi}
are different.
Moreover, the divisibility by a power of $2$ in the original Genocchi medians
(Proposition~\ref{Prop2.3}) is replaced by divisibility by a~power of
$(1+q)$.

\subsection[Inter-level statistic psi on colored interlacing triangles]{Inter-level statistic $\boldsymbol{\psi}$ on colored interlacing triangles}
\label{subsec:inter_level_statistic}

The probability measure on colored interlacing triangles coming from the central
limit theorem for LLT processes
\cite{ABW2023coloured_LLT}
is decomposed into a product of the top-row contribution (involving a~complicated vertex model with signed configuration weights, which we do not
pursue here\footnote{However, for depth $N=1$, this top-row vertex model yielding weights denoted by \smash{$g_{\Delta}^{c^{[N]}}$} in \cite{ABW2023coloured_LLT},
reduces to the well-known Mallows measure on $S_n$, under which a permutation
$\sigma$ is weighted proportionally to $q^{\mathrm{inv}(\sigma)}$.}) and simpler
inter-level contributions given by powers of $q$.
We focus on the latter: for each pair of consecutive levels in a colored
triangle, define a statistic $\psi\bigl(\lambda^{k}, \lambda^{k+1}\bigr)$, $k=1,\dots,N-1$:

\begin{Definition}[\cite{ABW2023coloured_LLT}]
	\label{def:psi} Fix $k=1,\dots,N-1$, and let $\lambda^{k+1}=\nu$ and
	$\lambda^{k}=\mu$ be two consecutive rows of a colored interlacing
	$n$-triangle $\boldsymbol\lambda$.
	Consider a one-row vertex model in which arrows can have one of $n$ colors,
	the horizontal line can carry an arbitrary composition of colors, and
	vertical lines can carry only one color:
	\begin{equation}
		\label{eq:psi_factor_vertex_model}
		\begin{tikzpicture}
			[scale=1.0,baseline=(current bounding box.center),>=stealth]
			\colorlet{lgray}{white!85!black} \draw[lgray,line width=4pt,<-]
			(0.5,0) -- (12,0); \foreach\x in {1,2,3,5,6,7,9,10,11}{
			\draw[lgray,line width=1.5pt,->] (\x,0) -- (\x,0.5); } \node[above]
			at (1,0.5) {$\nu[1]_1$}; \node[above] at (2,0.5) {$\cdots$};
			\node[above] at (3,0.5) {$\nu[1]_{k+1}$};
			\node[above] at (5,0.5) {$\cdots$}; \node[above] at (6,0.5)
			{$\cdots$}; \node[above] at (7,0.5) {$\cdots$};
			\node[above] at (9,0.5) {$\nu[n]_1$}; \node[above] at (10,0.5)
			{$\cdots$}; \node[above] at (11,0.5) {$\nu[n]_{k+1}$};
			\foreach\x in {1.5,2.5,5.5,6.5,9.5,10.5}{ \draw[lgray,line
			width=1.5pt,->] (\x,-0.5) -- (\x,0); }
			\node[below] at (1.5,-0.5) {$\mu[1]_1$}; \node[below] at (2,-0.3)
			{$\cdots$}; \node[below] at (2.5,-0.5) {$\mu[1]_k$};
			\node[below] at (5.5,-0.5) {$\cdots$}; \node[below] at (6.5,-0.5)
			{$\cdots$};
			\node[below] at (9.2,-0.5) {$\mu[n]_1$}; \node[below] at (10,-0.3)
			{$\cdots$}; \node[below] at (10.7,-0.5) {$\mu[n]_{k}$};
			\node[left] at (0.5,0) {$\varnothing$}; \node[right] at (12,0)
			{$[1,n]$};
		\end{tikzpicture}
	\end{equation}
	Here, the incoming configuration from the right has an arrow of each color
	$b=1,\dots,n$, the outgoing configuration on the left is empty, and the
	colors are conserved in the direction of arrow flow.
	That is, there are two possible types of vertices:
	\begin{equation}
		\label{eq:psi_factor_two_types_of_vertices}
		\begin{tikzpicture}
			[scale=1.25,baseline=(current bounding box.center),>=stealth]
			\colorlet{lgray}{white!85!black} \draw[lgray,line width=4pt,<-]
			(0.5,0) -- (1.5,0); \draw[lgray,line width=1.5pt,->] (1,-0.5) --
			(1,0); \node[right] at (1,-0.5) {$j$}; \node[left] at (0.5,0)
			{$A\cup \left\{ j \right\}$}; \node[right] at (1.5,0) {$A$}; \node
			at (1,0.5) {};
		\end{tikzpicture}
		\qquad \textnormal{and} \qquad
		\begin{tikzpicture}
			[scale=1.25,baseline=(current bounding box.center),>=stealth]
			\colorlet{lgray}{white!85!black} \draw[lgray,line width=4pt,<-]
			(0.5,0) -- (1.5,0); \draw[lgray,line width=1.5pt,->] (1,0) --
			(1,0.5); \node[left] at (0.5,0) {$A\setminus\left\{ j \right\}$};
			\node[right] at (1.5,0) {$A$}; \node[left] at (1,0.5) {$j$}; \node
			at (1,-0.5) {};
		\end{tikzpicture}
	\end{equation}
	Here $A\subset\left\{ 1,\dots,n  \right\}$ is an arbitrary subset of
	colors.
	Denote
	\begin{equation*}
		\psi(\mu,\nu) \coloneqq \sum_{p}\,
		\#\left( A_p\cap \left\{ \mu_p+1,\dots,n  \right\} \right),
	\end{equation*}
	where the sum is over the positions $p$ of the bottom row $\mu$ (one vertex of
	the first type in \eqref{eq:psi_factor_two_types_of_vertices} for each), $\mu_p$
	is the color inserted from below at $p$, and $A_p\subseteq\{1,\dots,n\}$ is the
	set of colors on the horizontal line immediately to the right of that vertex.
	In words, the $p$-th summand is the number of arrows of colors greater than
	$\mu_p$ which pass over~$\mu_p$ when it is inserted from below.

	The conservation of arrows implies that for fixed rows $\mu$ and $\nu$ which
	are interlacing in the sense of Definition~\ref{def:colored_interlacing_triangle},
	there is only one configuration of arrows in~\eqref{eq:psi_factor_vertex_model}.
\end{Definition}

For the colored interlacing triangle in Figure~\ref{fig:small_interlacing_example}, we
have $\psi(\lambda^1, \lambda^2) = 2$ and $\psi(\lambda^2, \lambda^3) = 3$.

\textbf{Combinatorial interpretation of $\boldsymbol{\psi}$.}
The statistic $\psi(\mu, \nu)$ from Definition~\ref{def:psi} admits an equivalent
description which does not involve a vertex model.
Consider the transition from level $k$ to level $k+1$.
For a position $p$ in the linear order \eqref{eq:interlacing_order} and a color
$c'$, let
\begin{align*}
	&T_{<p}(c')\coloneqq \#\{\text{level-$(k+1)$ positions of $c'$ that are to the left of $p$}\},\\
	&B_{<p}(c')\coloneqq \#\{\text{level-$k$ positions of $c'$ that are to the
	left of $p$}\}
\end{align*}
denote the number of occurrences of $c'$ in the top and bottom rows,
respectively, to the left of~$p$.
Then
\begin{equation}
	\label{eq:psi_combinatorial} \psi(\lambda^{k}, \lambda^{k+1}) =
	\sum_{\text{positions $p$ at level $k$}} \ssp\sum_{c' > c(p)}
	\bigl(T_{<p}(c') - B_{<p}(c')\bigr),
\end{equation}
where $c(p)$ is the color at position $p$ in level $k$.

For an example, consider the following colored interlacing $4$-triangle of
depth~$2$:
\begin{center}
	\begin{tikzpicture}
		[scale=0.75, dot1/.style={circle, fill=red!80, minimum size=5.5mm, inner
		sep=0pt, font=\small\bfseries}, dot2/.style={circle,
		fill=green!70!black, minimum size=5.5mm, inner sep=0pt,
		font=\small\bfseries, text=white}, dot3/.style={circle, fill=blue!70,
		minimum size=5.5mm, inner sep=0pt, font=\small\bfseries, text=white},
		dot4/.style={circle, fill=orange!70!black, minimum size=5.5mm, inner
		sep=0pt, font=\small\bfseries, text=white}] \def\h{0.8} \def\v{1.05}
		\def\tsep{3.5*\h}
		\node[left] at (-2.5*\h, 0) {$k=1$}; \node[left] at (-2.5*\h, \v)
		{$k=2$}; \node[left] at (14.5*\h, 0) {$\mu$}; \node[left] at (14.5*\h,
		\v) {$\nu$};
		\draw[gray, densely dashed] (-2*\h,0) -- (3*\tsep+1.5*\h,0); \draw[gray,
		densely dashed] (-2*\h,\v) -- (3*\tsep+1.5*\h,\v);
		\begin{scope}[shift={(0,0)}]
			\draw[gray!40, thick] (-\h,\v) -- (\h,\v) -- (0,0) -- cycle;
			\node[dot3] at (-\h,\v) {3}; \node[dot1] at (\h,\v) {1}; \node[dot3]
			at (0,0) {3}; \node at (0,\v+0.7) {$i=1$};
		\end{scope}
		\begin{scope}[shift={(\tsep,0)}]
			\draw[gray!40, thick] (-\h,\v) -- (\h,\v) -- (0,0) -- cycle;
			\node[dot4] at (-\h,\v) {4}; \node[dot2] at (\h,\v) {2}; \node[dot1]
			at (0,0) {1}; \node at (0,\v+0.7) {$i=2$};
		\end{scope}
		\begin{scope}[shift={(2*\tsep,0)}]
			\draw[gray!40, thick] (-\h,\v) -- (\h,\v) -- (0,0) -- cycle;
			\node[dot3] at (-\h,\v) {3}; \node[dot4] at (\h,\v) {4}; \node[dot4]
			at (0,0) {4}; \node at (0,\v+0.7) {$i=3$};
		\end{scope}
		\begin{scope}[shift={(3*\tsep,0)}]
			\draw[gray!40, thick] (-\h,\v) -- (\h,\v) -- (0,0) -- cycle;
			\node[dot1] at (-\h,\v) {1}; \node[dot2] at (\h,\v) {2}; \node[dot2]
			at (0,0) {2}; \node at (0,\v+0.7) {$i=4$};
		\end{scope}
	\end{tikzpicture}
\end{center}
Let us illustrate the computation of $\psi(\mu,\nu)$ via
\eqref{eq:psi_combinatorial}.
For each bottom-row position $p$ with color $c(p)$, we list the pairs
$(T_{<p}(c'),\, B_{<p}(c'))$ for each $c'>c(p)$:
\begin{alignat*}
	{2} c(p_1)=\textcolor{blue!70}{3}:     &\quad c'=4\colon\;(0,0),                                              &\qquad & \text{subtotal } 0;\\
	c(p_2)=\textcolor{red!80}{1}:          &\quad c'=2\colon\;(0,0),\  c'=3\colon\;(1,1),\
	c'=4\colon\;(1,0), &       & \text{subtotal } 1;\\
	c(p_3)=\textcolor{orange!70!black}{4}: &\quad \text{no }c'>4,                                                 &       & \text{subtotal } 0;\\
	c(p_4)=\textcolor{green!70!black}{2}:  &\quad c'=3\colon\;(2,1),\
	c'=4\colon\;(2,1),                      &       & \text{subtotal } 2.
\end{alignat*}
Hence $\psi(\mu,\nu) = 0+1+0+2 = 3$.

Define the \emph{color complement} involution acting on colors by
$\bar{c} = n+1-c$ for each $c \in \{1,\dots,n\}$, and extend it to sequences
entrywise.

\begin{Lemma}[color complement symmetry]
	\label{lem:color_complement} If $\mu = \lambda^k \prec \lambda^{k+1} = \nu$
	are interlacing rows in the sense of
	Definition~{\rm \ref{def:colored_interlacing_triangle}}, then $\bar{\mu} \prec \bar{\nu}$,
	and
	\[ \psi(\mu, \nu) + \psi(\bar{\mu}, \bar{\nu}) =
		k\binom{n}{2}.
	\]
\end{Lemma}
\begin{proof}
	The color complement map preserves the interlacing since it is simply a
	permutation of colors.
	Now, consider the vertex model in Definition~\ref{def:psi}.
	Since the arrow configuration for fixed $(\mu,\nu)$ is unique, the bar map
	carries it to the configuration for $(\bar{\mu},\bar{\nu})$, relabeling
	each color~$a$ by~$\bar{a}$ and keeping all positions unchanged.
	Consider a vertex of the first type in~\eqref{eq:psi_factor_two_types_of_vertices} with incoming color $j$, and
	let~$A$ be the set of colors on the horizontal line immediately to its
	right; since the occurrences of~$j$ in $\nu$ and $\mu$ alternate in the
	linear order \eqref{eq:interlacing_order}, we have $j\notin A$.
	This vertex contributes $\# \{ a\in A\mid a>j \}$ to
	$\psi(\mu,\nu)$, while the complemented vertex contributes
	$\# \{ a\in A \mid \bar{a}>\bar{j}  \}
	=\# \{ a\in A\mid a<j  \}$ to $\psi(\bar{\mu},\bar{\nu})$, and the
	two contributions sum to $\#A$.
	If our vertex is the $r$-th vertex of the first type in the $i$-th
	triangle, then strictly to its right the diagram contains
	$(n-i)(k+1)+(k+1-r)$ vertices of the second type and $(n-i)k+(k-r)$
	vertices of the first type; since the line carries all $n$ colors at its
	right edge, the conservation of arrows gives
	$\#A=n+[(n-i)k+(k-r)]-[(n-i)(k+1)+(k+1-r)]=i-1$.
	Summing over the $k$ vertices of the first type in each of the~$n$
	triangles, we obtain
\[
\psi(\mu, \nu) + \psi(\bar{\mu}, \bar{\nu})
	=k\sum_{i=1}^{n}(i-1)= k\binom{n}{2}.\tag*{\qed}
\]\renewcommand{\qed}{}
\end{proof}

In the remainder of this subsection, we specialize to depth $N=2$ and
write $\psi(\boldsymbol\lambda)\coloneqq\psi\bigl(\lambda^1,\lambda^2\bigr)$ for
$\boldsymbol\lambda\in\mathcal{T}_2(n)$.
Index the $n$ triangles by $1,\dots,n$.
The \emph{canonical} configuration $\lambda^{\mathrm{can}}$ has bottom row
$\mathrm{id}$ and, at every triangle $j$, both top dots equal to the color~$j$; from~\eqref{eq:psi_combinatorial} one checks
$\psi(\lambda^{\mathrm{can}})=0$.
The following definition and lemma isolate a gluing mechanism under which
$\psi$ is additive; we use it in Section~\ref{subsec:q_genocchi_case} to show
that $q^{\psi}$ is multiplicative over the interval decomposition of
Definition~\ref{def:indecomposable_triangle}, and in Section~\ref{sec:enumeration} to
bound the coefficients of the $q$-Genocchi medians from below
(Proposition~\ref{prop:leading_lower_bound}).

\begin{Definition}
	\label{def:block_defect}
	Let $B=\{p,p+1,\dots,p+w-1\}$ be a block of $w$ consecutive triangles. A~\emph{defect on $B$} is a valid configuration that agrees with
	$\lambda^{\mathrm{can}}$ on every triangle outside $B$ and is
	\emph{color-balanced} on $B$: for each color $c$, the number of top dots
	of color $c$ on~$B$, and the number of bottom dots of color $c$ on~$B$,
	equal those of $\lambda^{\mathrm{can}}$. In particular the only colors
	occurring on $B$ are its canonical colors $p,\dots,p+w-1$, and these colors
	do not occur outside $B$.
\end{Definition}

Within its block, the bottom row of a defect may be an arbitrary permutation of the colors $p,\dots,p+w-1$.  Defects thus generalize the blocks of the interval decomposition of Definition~\ref{def:indecomposable_triangle}, whose bottom rows are the identity permutations; this extra generality is used in Section~\ref{sec:enumeration}.

\begin{Lemma}\label{lem:defect_additivity}
	Let $B_1,\dots,B_k$ be pairwise disjoint blocks and, for each $r$, let
	$\Lambda_r$ be a defect on~$B_r$. Let $\Lambda$ be the configuration that
	agrees with $\Lambda_r$ on each $B_r$ and with $\lambda^{\mathrm{can}}$
	elsewhere. Then $\Lambda$ is a~valid colored interlacing triangle and
	\begin{equation}
		\label{eq:defect_additivity}
		\psi(\Lambda)=\sum_{r=1}^{k}\psi(\Lambda_r).
	\end{equation}
	Moreover, each $\psi(\Lambda_r)$ equals the $\psi$-value of the content
	of $\Lambda_r$ on its block, viewed as an element of
	$\mathcal{T}_2(w_r)$ after shifting the colors to $\{1,\dots,w_r\}$,
	where $w_r$ is the width of $B_r$.
\end{Lemma}

\begin{proof}
	Interlacing is imposed separately on each color: the two
	top dots of a color must straddle its bottom dot in the linear order
	\eqref{eq:interlacing_order}. By Definition~\ref{def:block_defect}, every color $c$ has
	all three of its dots either inside a single block $B_r$ (when $c$ is one of
	its colors) or entirely in the canonical region. In the first case $\Lambda$
	coincides near $c$ with the valid configuration $\Lambda_r$, in the second
	with $\lambda^{\mathrm{can}}$; in either case the interlacing constraint for
	$c$ holds. Hence $\Lambda$ is valid.

	Now, for a configuration $\Theta$ and a bottom position $p$
	(the bottom of triangle $p$) write
\[
f_\Theta(p)=\sum_{c'>c(p)} (T_{<p}(c')-B_{<p}(c') ),
		\qquad\text{so that}\quad \psi(\Theta)=\sum_{p}f_\Theta(p)
\]
	by \eqref{eq:psi_combinatorial}, where $c(p)$, $T_{<p}(c')$, $B_{<p}(c')$
	depend only on the part of $\Theta$ weakly to the left of $p$. A~direct
	evaluation gives $f_{\lambda^{\mathrm{can}}}(p)=0$ for every $p$, so
	$\psi(\lambda^{\mathrm{can}})=0$.

	Fix a bottom position $p$ and consider $\Lambda$. Any block $B_s$ lying
	entirely to the \emph{right} of $p$ contributes nothing to $T_{<p}(c')$ or
	$B_{<p}(c')$. Any block $B_s$ lying entirely to the \emph{left} of $p$ is
	color-balanced, so it contributes to each of $T_{<p}(c')$ and $B_{<p}(c')$
	exactly its canonical amount; replacing its content by the canonical one
	leaves $T_{<p}(c')$ and $B_{<p}(c')$ unchanged. Since the blocks are disjoint,
	$p$ meets at most one block $B_r$, and the data entering $f_\bullet(p)$ to the
	left of $p$ within $B_r$ is the same in $\Lambda$ and in $\Lambda_r$.
	Therefore,
	\[
		f_{\Lambda}(p)=f_{\Lambda_r}(p)\ \ (p\in B_r),
		\qquad
		f_{\Lambda}(p)=f_{\lambda^{\mathrm{can}}}(p)=0\ \ (p\notin\textstyle\bigcup_r B_r).
	\]
	The identical argument applied to each $\Lambda_r$ (canonical off $B_r$)
	gives $\psi(\Lambda_r)=\sum_{p\in B_r}f_{\Lambda_r}(p)$. Summing
	$f_{\Lambda}$ over all $p$ yields \eqref{eq:defect_additivity}.

	Finally, for $p\in B_r$ the summand $f_{\Lambda_r}(p)$ involves only
	the colors of $B_r$: the canonical colors to the left of $B_r$ are
	smaller than $c(p)$ and never enter the sum, while those to the right
	have $T_{<p}(c')=B_{<p}(c')=0$. Since shifting the colors to
	$\{1,\dots,w_r\}$ preserves their relative order, $\psi(\Lambda_r)$
	equals the $\psi$-value of the block content.
\end{proof}

\subsection[q-counting polynomial]{$\boldsymbol{q}$-counting polynomial}
\label{subsec:q_counting_polynomial}

We now define a $q$-deformation of the enumeration of colored interlacing
triangles $T_N(n)$.

\begin{Definition}
	\label{def:q_polynomial} For a colored interlacing $n$-triangle
	$\boldsymbol\lambda$ of depth $N$, define the \emph{total $\psi$-weight}
	\begin{equation*}
		\psi(\boldsymbol\lambda) \coloneqq \sum_{k=1}^{N-1} \psi\bigl(\lambda^k,
		\lambda^{k+1}\bigr).
	\end{equation*}
	The \emph{$q$-counting polynomial} is
	\begin{equation*}
		T_N(n;q) \coloneqq \sum_{\boldsymbol\lambda \in \mathcal{T}_N(n)}
		q^{\psi(\boldsymbol\lambda)}.
	\end{equation*}
	Clearly, at $q=1$, we have $T_N(n;1) = T_N(n) = \#\mathcal{T}_N(n)$.
\end{Definition}

\begin{Proposition}
	\label{prop:q_polynomial_properties} The polynomial $T_N(n;q)$ has the
	following properties:
	\begin{enumerate}[\rm (i)]\itemsep=0pt
		\item $\deg T_N(n;q) = \binom{n}{2} \cdot \binom{N}{2}$.
		\item $T_N(n;q) = q^{\binom{n}{2}\binom{N}{2}} T_N(n;1/q)$.
		\item For $N\ge2$, we have $2^{n-1} \mid T_N(n;q)$ as a polynomial
		(that is, all coefficients are divisible by $2^{n-1}$).
	\end{enumerate}
\end{Proposition}
\begin{proof}
(i) The maximum value of $\psi(\lambda^k, \lambda^{k+1})$ is $k\binom{n}{2}$
		(achieved when all colors $j' > j$ are active at each level-$k$ vertex
		of color $j$).
		Summing over $k=1,\dots,N-1$ gives
		$\sum_{k=1}^{N-1} k\binom{n}{2} = \binom{n}{2}\binom{N}{2}$.

(ii) The color complement involution $\bar{\boldsymbol\lambda}$
		is a bijection on $\mathcal{T}_N(n)$, and by
		Lemma~\ref{lem:color_complement}, we have
		$\psi(\boldsymbol\lambda) + \psi(\bar{\boldsymbol\lambda}) = \binom{n}{2}\binom{N}{2}$.
		This implies the palindromic property of $T_N(n;q)$.

(iii)
The $n-1$
independent involutions from the proof of Proposition~\ref{Prop2.3} preserve
the $\psi$-statistic. Indeed, since the two swapped top-level
entries are adjacent in the linear order~\eqref{eq:interlacing_order}, with no
level-$(N-1)$ entry between them, the swap does not change the
top-level colors lying to the left of any level-$(N-1)$ position,
and hence preserves the $\psi$-statistic.
Thus, all coefficients of $T_N(n;q)$ inherit the divisibility by
		$2^{n-1}$.
\end{proof}

\subsection[q-enumeration in the Genocchi case]{$\boldsymbol{q}$-enumeration in the Genocchi case}
\label{subsec:q_genocchi_case}

For depth $N=2$, the polynomial $T_2(n;q)$ is a $q$-analog of the quantity
$T_2(n) = n! \cdot H_n$ from Theorem~\ref{thm:Genocchi_counts}.
By Proposition~\ref{prop:q_polynomial_properties}, $T_2(n;q)$ is palindromic of degree
$\binom{n}{2}$ and divisible by $2^{n-1}$.
The first few values are
\begin{align*}
	&T_2(1;q)= 1, \\
	&T_2(2;q)= 2(1+q), \\
	&T_2(3;q)= 2^2\bigl(1+5q+5q^2+q^3\bigr), \\
	&T_2(4;q)= 2^3\bigl(1+10q+47q^2+52q^3+47q^4+10q^5+q^6\bigr), \\
	 &T_2(5;q)=
	2^4\bigl(1+15q+132q^2+527q^3+1019q^4+1172q^5+\textnormal{palindrome}\bigr).
\end{align*}
See also Section~\ref{subsec:q_data} for further values.
We notice the linear growth of the coefficient of the first power of~$q$.

\begin{Proposition}[linear coefficient]\label{prop:linear_coefficient} Writing
	$T_2(n;q) = 2^{n-1}(1 + a_1 q + \cdots)$, the linear coefficient satisfies
	$a_1 = 5(n-2)$ for $n \ge 3$.
\end{Proposition}
\begin{proof}
	We need to count the number of depth-2 colored interlacing triangles with
	$\psi = 1$.
	Let $\lambda^1=\sigma\in S_n$ denote the permutation in the bottom row.
	We claim that the number of inversions of $\sigma$ must be either $0$ or $1$
	to have $\psi = 1$.
	Modulo this claim, let us first finish the computation of $a_1$ for $n\ge3$.
	We have the following contributions (normalized by~$2^{n-1}$, the size of each orbit under the
top-row involutions, see Proposition~\ref{Prop2.3}):{\samepage
	\begin{itemize}\itemsep=0pt
		\item
		When $\sigma=\mathrm{id}$ (no inversions), we can achieve $\psi=1$ by
		choosing one of the middle $n-2$ bottom positions $c=2,\dots,n-1$ to
		contribute $1$ to $\psi$.
		This implies that the color crossing this position must be exactly
		$c+1$, and forces the configuration, up to top-row involutions.
		\item
		When $\sigma$ is a boundary transposition $(1,2)$ or $(n-1,n)$, there
		are two ways to achieve $\psi=1$ (up to top-row involutions).
		Indeed, for example, the triangle with $\sigma=(2,1,3,4,\dots,n) $ may
		have top row either $(2| 12| 13| 34| 45| \dots) $, or
		$(2| 13| 12| 34| 45| \dots) $, where the vertical bars indicate where
		the bottom row entries are inserted, and two entries between bars are
		unordered (due to top-row involutions).
		\item
		Finally, when $\sigma$ is one of the $n-3$ internal transpositions
		$(i,i+1)$, $2\le i \le n-2$, one similarly sees that there are four ways
		to achieve $\psi=1$ (up to top-row involutions).
	\end{itemize}
	The final count is $n-2+4+4(n-3)=5(n-2)$, as claimed.}

	It remains to prove the initial claim that the number of inversions in
	$\sigma$ is at most $1$ when $\psi=1$.
	For $i=1,\dots,n $, let $c_i$ denote the contribution of the $i$-th
	bottom dot to $\psi$ in Definition~\ref{def:psi}, so that $\psi=\sum_{i=1}^{n}c_i$.
	The conservation-of-arrows count in the proof of
	Lemma~\ref{lem:color_complement} shows that the set of colors active at this
	dot, together with the inserted color $\sigma_i$, has exactly $i$
	elements.
	The $i-c_i$ of them not exceeding $\sigma_i$ are distinct colors from
	$\{1,\dots,\sigma_i\}$, whence $\sigma_i \ge i-c_i$ for all $i$.
	If $\psi\le1$, then $c_i=0$ for all $i$ except at most one position $p$
	with $c_p=1$, so that $\sigma_i\ge i$ for $i\ne p$ and
	$\sigma_p\ge p-1$.
	Since $\sum_{i}\sigma_i=\sum_{i}i$, either $\sigma=\mathrm{id}$, or
	$\sigma_p=p-1$, $\sigma_q=q+1$ for exactly one $q\ne p$, and
	$\sigma_i=i$ for all other $i$.
	In the latter case, comparing the values $\{p-1,q+1\}$ with the values
	$\{p,q\}$ missing from the fixed positions forces $q=p-1$, so $\sigma$
	is the adjacent transposition of $p-1$ and $p$.
	In both cases $\mathrm{inv}(\sigma)\le1$, proving the claim.

	Finally, let us justify writing $T_2(n;q)=2^{n-1}(1+a_1q+\cdots)$, that
	is, that the constant term of $T_2(n;q)$ is exactly $2^{n-1}$
	(Proposition~\ref{prop:q_polynomial_properties} yields only its divisibility by
	$2^{n-1}$).
	If $\psi=0$, then all $c_i=0$, and the inequality above gives
	$\sigma=\mathrm{id}$.
	By \eqref{eq:Hid_Sfrac} below
	(which does not rely on the present statement),
	the generating function of the triangles
	with identity bottom row weighted by $q^{\psi}$ specializes at $q=0$
	(where $\alpha_1=\alpha_2=1$ and all other $\alpha_j$ vanish) to
	$\frac{1-x}{1-2x}=1+\sum_{n\ge1}2^{n-1}x^{n}$.
	Hence there are exactly $2^{n-1}$ triangles with $\psi=0$, as desired.
	This completes the proof.
\end{proof}

In Section~\ref{subsec:conjectures_on_coefficients}, we conjecture expressions
for the next few coefficients based on the computed values of $T_2(n;q)$.
\begin{Remark}
	\label{rem:psi_vs_inv} The proof of Proposition~\ref{prop:linear_coefficient} shows
	that $\psi \ge 2$ whenever $\mathrm{inv}(\sigma) \ge 2$, but the inequality
	$\psi \ge \mathrm{inv}(\sigma)$ does \emph{not} hold in general.
	For example, with $n=3$ and $\sigma = (3,2,1)$, we have
	$\mathrm{inv}(\sigma) = 3$, yet the top row $(3|12|23|1)$ gives $\psi = 2$.
	See also Lemma~\ref{lem:psi_crossterm} below.
\end{Remark}

The polynomials $T_2(n;q)$ are $q$-analogs of $n! \cdot H_n$, and refining them
over the bottom row yields a $q$-analog of the Genocchi median $H_n$ for each
permutation $\sigma\in S_n$.

\begin{Definition}
	\label{def:H_sigma} For any permutation $\sigma \in S_n$, the
	\emph{$q$-Genocchi median conditioned on $\sigma$} is given~by
	\begin{equation*}
		H_n^\sigma(q) \coloneqq \sum_{\boldsymbol\lambda \in
		\mathcal{T}_2(n),\ \lambda^1 = \sigma} q^{\psi(\lambda^1, \lambda^2)}.
	\end{equation*}
\end{Definition}

For example, for the identity permutation $\mathrm{id} = (1,2,\dots,n)$, we
get
\begin{align*}
	&H_1^{\mathrm{id}}(q)= 1, \\
	&H_2^{\mathrm{id}}(q)= 2, \\
	&H_3^{\mathrm{id}}(q)= 4(1+q), \\
	&H_4^{\mathrm{id}}(q)= 8\bigl(1+  2q + 4q^2\bigr), \\
	&H_5^{\mathrm{id}}(q)= 16\bigl(1 + 3q + 9q^2 + 16q^3 + 9q^4\bigr).
\end{align*}

\begin{Proposition}\label{prop:H_sigma_properties} The polynomials $H_n^\sigma(q)$ have the
	following properties:
	\begin{enumerate}[\rm (i)]\itemsep=0pt
		\item $2^{n-1} \mid H_n^\sigma(q)$ for all $\sigma \in S_n$.
		\item
		$H_n^{\bar\sigma}(q) = q^{\binom{n}{2}} H_n^\sigma(1/q)$, where
		$\bar\sigma$ denotes the color complement
		$\bar\sigma(i) = n+1-\sigma(i)$.
		\item $H_n^{\sigma}(1)=H_n$, the $n$-th Genocchi median, for all $\sigma \in S_n$.
	\end{enumerate}
\end{Proposition}
\begin{proof}
(i) The $n-1$ independent involutions from Proposition~\ref{Prop2.3}
		preserve both the bottom row and the $\psi$-statistic.

(ii) By Lemma~\ref{lem:color_complement}, the color complement bijection
		$\boldsymbol\lambda \mapsto \bar{\boldsymbol\lambda}$ maps triangles
		with bottom row $\sigma$ to triangles with bottom row $\bar\sigma$, and
		satisfies
		$\psi(\bar{\boldsymbol\lambda}) = \binom{n}{2} - \psi(\boldsymbol\lambda)$.

(iii) When $q=1$, we are just counting the number of colored interlacing triangles
		with a fixed bottom row $\sigma$, and permutations of colors do not
		affect the interlacing constraints.
\end{proof}

\begin{Remark}
	\label{rem:inverse_not_palindromic} The inversion map
	$\sigma \mapsto \sigma^{-1}$ does \emph{not} induce a relation between
	$H_n^\sigma(q)$ and $H_n^{\sigma^{-1}}(q)$ similar to the color complement.
	For instance, $\sigma = (1,3,4,2)$ has inverse $\sigma^{-1} = (1,4,2,3)$,
	and $H_4^{(1,3,4,2)}(q) = 8\bigl(4q^2 + 3q^3\bigr)$ while
	$H_4^{(1,4,2,3)}(q) = 8\bigl(5q^2 + 2q^3\bigr)$.
\end{Remark}

\begin{Remark}
	\label{rem:H_sigma_not_known} One can check that the polynomials
	$H_n^\sigma(q)$ do not coincide with the known $q$-analogs of Genocchi
	medians from Section~\ref{subsec:known_q_analogs}.
	For instance, the polynomials $H_n^\sigma(q)$ are divisible by $2^{n-1}$,
	while the ones from Section~\ref{subsec:known_q_analogs} instead contain certain
	powers of $(1+q)$.
	Even removing these prefactors, the resulting polynomials do not generally
	match for $n\ge4$.
\end{Remark}

\subsection[Interval decomposition of the q-enumeration]{Interval decomposition of the $\boldsymbol{q}$-enumeration}
\label{subsec:q_interval_decomposition}

We now show that the statistic $\psi$ respects the interval decomposition of
Definition~\ref{def:indecomposable_triangle}.
We first treat the identity bottom row, where $\psi$ is simply additive
over the blocks, and then arbitrary bottom rows, where the additivity
acquires an explicit inversion cross-term.

\begin{Proposition}
	\label{prop:psi_multiplicative}
	For the identity bottom row,
	the statistic $\psi$ is additive over the interval decomposition of
	Definition~{\rm \ref{def:indecomposable_triangle}}.
	Namely, if
	$\boldsymbol\lambda\in\mathring{\mathcal{T}}_2(n)$ has blocks
	$\boldsymbol\lambda_1,\dots,\boldsymbol\lambda_k$, then
	$\psi(\boldsymbol\lambda)
	=\psi(\boldsymbol\lambda_1)+\cdots+\psi(\boldsymbol\lambda_k)$.
	Consequently, denoting by $\mathrm{Irr}_n(q)$ the sum of $q^{\psi}$
	over the indecomposable elements of $\mathring{\mathcal{T}}_2(n)$, and
	setting $H_0^{\mathrm{id}}(q)\coloneqq1$, we have
	\begin{equation}
		\label{eq:first_return_identity_bottom}
		\sum_{n\ge0} H_n^{\mathrm{id}}(q)\ssp x^n =
		\biggl(1-\sum_{j\ge1} \mathrm{Irr}_j(q)\ssp x^j\biggr)^{-1}.
	\end{equation}
\end{Proposition}
\begin{proof}
The additivity is a specialization of Lemma~\ref{lem:defect_additivity} to the identity bottom row. The identity~\eqref{eq:first_return_identity_bottom} follows by the standard interval decomposition generating function formula for sums over interval partitions, see the proof of Corollary~\ref{cor:indecomposable_count}.
\end{proof}

By Proposition~\ref{prop:pure_indecomposable_bijection},
$\mathrm{Irr}_n(1)=I_n=h^\flat_n$, so the polynomials $\mathrm{Irr}_n(q)$
provide a $q$-analog of the indecomposable counting
sequence $h^\flat$ (see
Remark~\ref{rem:indecomposable_triangles} below for explicit values).

For a general bottom row, the statistic $\psi$ is \emph{not} additive
over splittings. For example, for $n=2$ and $\sigma=(2,1)$, the triangle with top row
$(2|21|1)$ splits after the first triangle but has $\psi=1$, while both
of its blocks have $\psi=0$.
We now show that the failure of additivity is explicit,
which produces a
$q$-exponential version of
the generating function identity \eqref{eq:first_return_counting} for arbitrary bottom rows.

\begin{Definition}
	\label{def:indecomposable_general}
	Let $\boldsymbol\lambda\in\mathcal{T}_2(n)$.
	As in Definition~\ref{def:indecomposable_triangle}, we say that
	$\boldsymbol\lambda$ \emph{splits} after the $j$-th triangle
	($1\le j\le n-1$) if no color occurs both among the first $j$ triangles
	and after them, and we call $\boldsymbol\lambda$ \emph{indecomposable}
	if it does not split.
	For $\boldsymbol\lambda\in\mathring{\mathcal{T}}_2(n)$ this is
	Definition~\ref{def:indecomposable_triangle}.
	Let $\widetilde{\mathrm{Irr}}_n(q)$ denote the sum of $q^{\psi}$ over
	the indecomposable elements of $\mathcal{T}_2(n)$.
\end{Definition}

\begin{Lemma}
	\label{lem:splitting_structure}
	A triangle $\boldsymbol\lambda\in\mathcal{T}_2(n)$ splits after the
	$j$-th triangle if and only if no interlacing arc crosses the division
	between the $j$-th and the $(j+1)$-st triangles.
	In this case, exactly $j$ colors occur among the first $j$ triangles,
	forming a subset $S\subseteq\{1,\dots,n\}$ with $|S|=j$, and the
	restrictions of $\boldsymbol\lambda$ to the first $j$ and the last
	$n-j$ triangles, with the colors relabeled via the increasing
	bijections $S\to\{1,\dots,j\}$ and
	$\{1,\dots,n\}\setminus S\to\{1,\dots,n-j\}$, are elements of~$\mathcal{T}_2(j)$ and $\mathcal{T}_2(n-j)$, respectively.
	Consequently, cutting at every division after which
	$\boldsymbol\lambda$ splits decomposes it uniquely into consecutive
	blocks with color sets $S_1,\dots,S_r$ $($listed left to right, so that
	$\{1,\dots,n\}=S_1\sqcup\cdots\sqcup S_r)$ and indecomposable
	standardized blocks $\boldsymbol\lambda_1,\dots,\boldsymbol\lambda_r$.
\end{Lemma}
\begin{proof}
	If $\boldsymbol\lambda$ splits after the $j$-th triangle, then in
	particular no arc crosses the division.
	Conversely, suppose that no arc crosses the division.
	By interlacing, the level-$1$ dot of every color lies between its two
	level-$2$ dots in the linear order \eqref{eq:interlacing_order}. Hence
	all three dots of every color lie on one side of the division, and
	$\boldsymbol\lambda$ splits.

	Now suppose that $\boldsymbol\lambda$ splits after the $j$-th
	triangle.
	The first $j$ triangles contain $2j$ level-$2$ dots, and every color
	occurring there contributes exactly two of them (all of its dots lie on
	one side). Hence exactly $j$ colors occur among the first $j$
	triangles, and their level-$1$ dots fill the bottom positions
	$1,\dots,j$.
	Both restrictions have the correct number of dots of each color, and
	the interlacing condition for each color constrains only the relative
	order of its three dots, which is preserved under restriction to
	consecutive triangles and under order-preserving relabeling of the
	colors.
	Finally, each standardized block is indecomposable: by the arc
	criterion above, a~splitting of a block would produce a splitting of
	$\boldsymbol\lambda$ at the corresponding intermediate division, since
	the arcs of the colors of the other blocks do not cross it.
\end{proof}

\begin{Lemma}
	\label{lem:psi_crossterm}
	Let $\boldsymbol\lambda\in\mathcal{T}_2(n)$ have color sets
	$S_1,\dots,S_r$ and indecomposable standardized blocks
	$\boldsymbol\lambda_1,\dots,\boldsymbol\lambda_r$ as in
	Lemma~{\rm \ref{lem:splitting_structure}}.
	Then
	\begin{equation*}
		\psi(\boldsymbol\lambda)
		=\mathrm{inv}(S_1,\dots,S_r)
		+\sum_{i=1}^{r}\psi(\boldsymbol\lambda_i),
	\end{equation*}
	where
	$\mathrm{inv}(S_1,\dots,S_r)
	\coloneqq
	\# \{(a,b)\mid a\in S_i,\, b\in S_{i'},\, i<i',\, a>b \}$
	is the number of color inversions between distinct blocks.
\end{Lemma}
\begin{proof}
	As in the proof of Lemma~\ref{lem:defect_additivity}, for a bottom position
	$p$ carrying the color $c(p)$, write
	$f(p)=\sum_{c'>c(p)} (T_{<p}(c')-B_{<p}(c') )$, so that
	$\psi(\boldsymbol\lambda)=\sum_{p}f(p)$ by
	\eqref{eq:psi_combinatorial}.
	Fix $p$ in the $i'$-th block and a color $c'>c(p)$, and consider the
	three possible locations of~$c'$.
	If $c'$ belongs to a block strictly to the left of the $i'$-th, then
	all three dots of $c'$ lie strictly to the left of~$p$, whence
	$T_{<p}(c')-B_{<p}(c')=2-1=1$.
	If~$c'$ belongs to a block strictly to the right, then all its dots
	lie to the right of $p$, and $T_{<p}(c')-B_{<p}(c')=0-0=0$.
	If~$c'$ belongs to the $i'$-th block, then the corresponding summand
	coincides with its counterpart for the standardized block
	$\boldsymbol\lambda_{i'}$: the relabeling of the colors is
	order-preserving, so the condition $c'>c(p)$ is unaffected, and so are
	the dots of $c'$ to the left of $p$.
	Summing the contributions of the first type over all $p$ and $c'$
	counts exactly the pairs $(a,b)=(c',c(p))$ with $a$ in an earlier
	block than $b$ and $a>b$, that is, $\mathrm{inv}(S_1,\dots,S_r)$.
	The contributions of the third type sum to
	$\sum_{i}\psi(\boldsymbol\lambda_i)$, again by~\eqref{eq:psi_combinatorial}.
\end{proof}

For the $n=2$ example above, $\mathrm{inv}(\{2\},\{1\})=1$, in agreement
with $\psi=1$.

Summing $q^{\psi}$ over all of $\mathcal{T}_2(n)$ now yields a $q$-analog
of the generating function identity~\eqref{eq:first_return_counting} for the
full $q$-enumeration.
Recall the standard $q$-notation
$[k]_q\coloneqq 1+q+\cdots+q^{k-1}$,
$[k]_q!\coloneqq[1]_q\ssp[2]_q\cdots[k]_q$, and the Gaussian multinomial
coefficients
\smash{${n\brack n_1,\dots,n_r}_q\coloneqq
\frac{[n]_q!}{[n_1]_q!\cdots[n_r]_q!}$}.

\begin{Theorem}\label{thm:q_exponential_first_return}
	For every $n\ge1$, we have
	\begin{equation}\label{eq:q_composition}
		T_2(n;q)
		=\sum_{r\ge1}\, \sum_{\substack{n_1+\cdots+n_r=n\\ n_i\ge1}}
		{n\brack n_1,\dots,n_r}_q\,
		\prod_{i=1}^{r}\widetilde{\mathrm{Irr}}_{n_i}(q).
	\end{equation}
	Equivalently, setting $T_2(0;q)\coloneqq1$, we have an identity of
	formal power series in $x$:
	\begin{equation}\label{eq:q_exponential_first_return}
		\sum_{n\ge0}\frac{T_2(n;q)}{[n]_q!}\ssp x^n
		=\biggl(1-\sum_{j\ge1}
		\frac{\widetilde{\mathrm{Irr}}_j(q)}{[j]_q!}\ssp x^j\biggr)^{-1}.
	\end{equation}
\end{Theorem}
\begin{proof}
	By Lemma~\ref{lem:splitting_structure}, the map sending a triangle to the
	pair (ordered set partition of the colors, tuple of standardized
	blocks) is a bijection
	\begin{equation*}
		\mathcal{T}_2(n)\,\longleftrightarrow\,
		\bigsqcup_{r\ge1}\,
		\bigsqcup_{\substack{n_1+\cdots+n_r=n\\ n_i\ge1}}\,
		\bigsqcup_{\substack{\{1,\dots,n\}=S_1\sqcup\cdots\sqcup S_r\\
		|S_i|=n_i}}
		\bigl\{(\boldsymbol\lambda_1,\dots,\boldsymbol\lambda_r)\mid
		\boldsymbol\lambda_i\in\mathcal{T}_2(n_i)
		\text{ indecomposable}\bigr\}.
	\end{equation*}
	Indeed, the inverse map places the blocks side by side and restores
	the colors of the $i$-th block via the increasing bijection
	$\{1,\dots,n_i\}\to S_i$; the result is a valid triangle (the level
	counts are clear, and interlacing for each color is preserved under
	shifting of the triangle indices and order-preserving recoloring),
	and by the arc criterion of Lemma~\ref{lem:splitting_structure} it splits
	exactly at the $r-1$ block boundaries, so the two maps are mutually
	inverse.

	By Lemma~\ref{lem:psi_crossterm}, under this bijection,
	\begin{equation*}
		T_2(n;q)
		=\sum_{r\ge1}\, \sum_{\substack{n_1+\cdots+n_r=n\\ n_i\ge1}}
		\Biggl(\ \sum_{\substack{\{1,\dots,n\}=S_1\sqcup\cdots\sqcup
		S_r\\ |S_i|=n_i}}
		q^{\mathrm{inv}(S_1,\dots,S_r)}\Biggr)
		\prod_{i=1}^{r}\widetilde{\mathrm{Irr}}_{n_i}(q).
	\end{equation*}
	Encoding an ordered set partition $(S_1,\dots,S_r)$ with $|S_i|=n_i$
	by the word $u_1\cdots u_n$ with $u_k=i$ for $k\in S_i$ turns
	$\mathrm{inv}(S_1,\dots,S_r)$ into the number of inversions of a
	word with $n_i$ letters $i$; therefore, by MacMahon's classical
	theorem on inversions of multiset permutations \cite{MacMahonBook}
	(see also \cite[Section~1.3]{Stanley1997}), the inner sum equals
 ${n\brack n_1,\dots,n_r}_q$, which proves \eqref{eq:q_composition}.
	Dividing \eqref{eq:q_composition} by~$[n]_q!$ and summing over
	$n\ge0$ yields \eqref{eq:q_exponential_first_return}, since the
	coefficient of $x^n$ in the $r$-th power of
	\smash{$\sum_{j\ge1}\widetilde{\mathrm{Irr}}_j(q)\ssp x^j/[j]_q!$} equals
	\smash{$\sum_{n_1+\cdots+n_r=n}\prod_{i=1}^r
	\widetilde{\mathrm{Irr}}_{n_i}(q)/[n_i]_q!$}.
\end{proof}

\begin{Remark}
Recalling that $H_n=T_2(n)/n!$, for $q=1$ \eqref{eq:q_exponential_first_return} straightforwardly reduces to~\eqref{eq:first_return_counting}.
\end{Remark}
\begin{Corollary}
	\label{cor:q_analog_full_hflat}
	For all $n\ge1$, we have
	$\widetilde{\mathrm{Irr}}_n(1)=n!\ssp h^\flat_n$.
	Thus the polynomials $\widetilde{\mathrm{Irr}}_n(q)$ form a second
	$q$-analog, this time of the sequence $\{n!\ssp h^\flat_n\}$,
	different from the polynomials $\mathrm{Irr}_n(q)$ of
	Proposition~{\rm \ref{prop:psi_multiplicative}}: for instance,
	\smash{$\widetilde{\mathrm{Irr}}_3(q)=1+14q+14q^2+q^3$}, while
	$\mathrm{Irr}_3(q)=1+4q$.
\end{Corollary}
\begin{proof}
	Setting $q=1$ in \eqref{eq:q_exponential_first_return} and using
	$[n]_1!=n!$ and $T_2(n;1)=n!\cdot H_n$ turns it into
	$\sum_{n\ge0}H_n\ssp x^n
	=\bigl(1-\sum_{j\ge1}(\widetilde{\mathrm{Irr}}_j(1)/j!)\ssp
	x^j\bigr)^{-1}$.
	The coefficients of a series $J(x)$ with
	$\sum_{n\ge0}H_n\ssp x^n=(1-J(x))^{-1}$ are determined uniquely
	(recursively, $J_n=H_n-\sum_{j=1}^{n-1}J_j\ssp H_{n-j}$), so
	comparing with~\eqref{eq:first_return_counting} and using
	Corollary~\ref{cor:indecomposable_count} and Proposition~\ref{prop:pure_indecomposable_bijection}
	gives $\widetilde{\mathrm{Irr}}_n(1)/n!=I_n=h^\flat_n$.
\end{proof}

See Remark~\ref{rem:indecomposable_triangles} below for explicit values of
$\widetilde{\mathrm{Irr}}_n(q)$.

\subsection{Lower crossings and lower nestings on Dumont derangements}
\label{subsec:crossings_nestings}

Here we show that over the identity bottom row, the
statistic $\psi$ has a~description in terms of classical
permutation statistics on Dumont derangements.

Let $\boldsymbol\lambda\in\mathring{\mathcal{T}}_2(n)$ and let $\sigma\in
S_{2n}$ be the associated Dumont derangement (Theorem~\ref{thm:Genocchi_counts});
thus, for each color $c$, the two level-$2$ occurrences of $c$ lie in the
slots $\sigma(2c)<\sigma(2c-1)$, and $\sigma(2c)\le 2c-1<2c\le\sigma(2c-1)$.
Following Deb and Sokal \cite{DebSokal2024} (see also \cite{SokalZeng2022} and the references therein), call a pair $1\le i<j\le n$ a
\emph{lower crossing} if $\sigma(2i)<\sigma(2j)<2i$ and a \emph{lower
nesting} if $\sigma(2j)<\sigma(2i)<2i$, and write
$\operatorname{lcr}(\sigma)$ and $\operatorname{lne}(\sigma)$ for the number
of each.

\begin{Proposition}
	\label{prop:psi_crossings_nestings}
	For $\boldsymbol\lambda\in\mathring{\mathcal{T}}_2(n)$ with associated
	Dumont derangement $\sigma$,
	\begin{equation}
		\label{eq:psi_lcross_lnest}
		\psi(\boldsymbol\lambda)=
		\operatorname{lcr}(\sigma)+\operatorname{lne}(\sigma).
	\end{equation}
\end{Proposition}
Figure~\ref{fig:upper_lower_example} illustrates the counting of lower
crossings and lower nestings.
\begin{proof}[Proof of Proposition~\ref{prop:psi_crossings_nestings}]
	We use the cumulative form \eqref{eq:psi_combinatorial},
	$\psi(\boldsymbol\lambda)=\sum_p f(p)$, where
	\[
		f(p)=\sum_{c'>c(p)}\bigl(T_{<p}(c')-B_{<p}(c')\bigr).
	\]
	Since the bottom row is the identity, the marker at position $p$
	carries color $c(p)=p$ and sits between slots $2p-1$ and $2p$. For a
	color $c'>p$, the marker of $c'$ lies at position $c'>p$, so
	$B_{<p}(c')=0$.
	Thus $f(p)=\sum_{c'>p}T_{<p}(c')$ counts the level-$2$ dots of colors
	$c'>p$ lying in slots $1,\dots,2p-1$.
	For $c'>p$, the right slot satisfies $\sigma(2c'-1)>2c'-1\ge 2p+1$, so
	only the left slot $\sigma(2c')$ can lie below $2p$. Hence
	$T_{<p}(c')= [ \sigma(2c')<2p ]$ and
	\begin{equation*}
		\psi(\boldsymbol\lambda)
		=\sum_{1\le i<j\le n} [ \sigma(2j)<2i ].
	\end{equation*}
	For such a pair $\sigma(2i)<2i$ as well, so either
	$\sigma(2i)<\sigma(2j)<2i$ (a lower crossing) or
	$\sigma(2j)<\sigma(2i)<2i$ (a lower nesting). The two cases are
	exclusive and exhaustive, giving \eqref{eq:psi_lcross_lnest}.
\end{proof}

\begin{figure}[!ht]
	\centering
	\begin{tikzpicture}[
		scale=0.78,
		dot1/.style={circle, fill=red!80, minimum size=5mm, inner sep=0pt, font=\footnotesize\bfseries},
		dot2/.style={circle, fill=green!70!black, minimum size=5mm, inner sep=0pt, font=\footnotesize\bfseries, text=white},
		dot3/.style={circle, fill=blue!70, minimum size=5mm, inner sep=0pt, font=\footnotesize\bfseries, text=white},
		dot4/.style={circle, fill=orange!70!black, minimum size=5mm, inner sep=0pt, font=\footnotesize\bfseries, text=white}]
		\node[font=\footnotesize] at (-0.3,1.4) {$\lambda^2$};
		\node[font=\footnotesize] at (-0.3,0) {$\lambda^1$};
		\node[dot1] (t1) at (1,1.4)  {1};
		\node[dot2] (t2) at (3,1.4)  {2};
		\node[dot4] (t3) at (4,1.4)  {4};
		\node[dot3] (t4) at (6,1.4)  {3};
		\node[dot1] (t5) at (7,1.4)  {1};
		\node[dot2] (t6) at (9,1.4)  {2};
		\node[dot3] (t7) at (10,1.4) {3};
		\node[dot4] (t8) at (12,1.4) {4};
		\node[dot1] (b1) at (2,0)  {1};
		\node[dot2] (b2) at (5,0)  {2};
		\node[dot3] (b3) at (8,0)  {3};
		\node[dot4] (b4) at (11,0) {4};
		\draw[red!80,          line width=1pt] (t1) to[bend left=28] (t5);
		\draw[green!70!black,  line width=1pt] (t2) to[bend left=28] (t6);
		\draw[blue!70,         line width=1pt] (t4) to[bend left=40] (t7);
		\draw[orange!70!black, line width=1pt] (t3) to[bend left=22] (t8);
		\foreach \n/\p in {t1/1,t2/2,t3/3,t4/4,t5/5,t6/6,t7/7,t8/8}
			\node[below=2pt of \n, font=\scriptsize] {$\p$};
	\end{tikzpicture}

	\smallskip

	\begin{tikzpicture}[scale=0.78, every node/.style={font=\footnotesize}]
		\node[anchor=east] at (0.4,0.55) {$i=$};
		\node[anchor=east] at (0.4,-0.2) {$\sigma(i)=$};
		\foreach \i in {1,...,8} \node at (\i,0.55) {$\i$};
		\foreach \i/\s/\col in {1/5/red!80,2/1/red!80,3/6/green!70!black,4/2/green!70!black,%
			5/7/blue!70,6/4/blue!70,7/8/orange!70!black,8/3/orange!70!black}
			\node[text=\col, font=\footnotesize\bfseries] at (\i,-0.2) {$\s$};
	\end{tikzpicture}
	\caption{A depth-$2$ colored interlacing $4$-triangle with identity bottom
	row $\lambda^1=(1,2,3,4)$ and top row $\lambda^2=(1,2,4,3,1,2,3,4)$, and the
	associated Dumont derangement $\sigma=(5,1,6,2,7,4,8,3)\in S_8$; as in
	Figure~\ref{fig:bijection_dumont}, the arcs join the two level-$2$ occurrences of
	each color.
	Here the pair $(i,j)=(2,4)$ is a lower crossing, as
	$\sigma(4)=2<\sigma(8)=3<4$, and the pair $(3,4)$ is a lower nesting, as
	$\sigma(8)=3<\sigma(6)=4<6$. Thus
	$\operatorname{lcr}(\sigma)=\operatorname{lne}(\sigma)=1$, and
	$\psi(\boldsymbol\lambda)=2$ by \eqref{eq:psi_lcross_lnest}.}
	\label{fig:upper_lower_example}
\end{figure}

Write $\mathfrak{D}_{2n}$ for the set of Dumont derangements in $S_{2n}$.
Combining~\eqref{eq:psi_lcross_lnest} with Proposition~\ref{prop:splits_direct_sum}
(indecomposable triangles correspond to indecomposable Dumont
derangements) gives
\begin{equation}
	\label{eq:Hid_Irr_crossnest}
	H_n^{\mathrm{id}}(q)
	=\!\!\sum_{\sigma\in\mathfrak{D}_{2n}}\!\!
	q^{\operatorname{lcr}(\sigma)+\operatorname{lne}(\sigma)},
	\qquad
	\mathrm{Irr}_n(q)
	=\!\!\sum_{\substack{\sigma\in\mathfrak{D}_{2n}\\ \text{indecomposable}}}\!\!
	q^{\operatorname{lcr}(\sigma)+\operatorname{lne}(\sigma)}.
\end{equation}
The left-hand enumerator $H_n^{\mathrm{id}}(q)$ in
\eqref{eq:Hid_Irr_crossnest} is the specialization of the Deb--Sokal
crossing--nesting polynomial for Dumont derangements \cite{DebSokal2024} in
which $q$ marks lower crossings and lower nestings.

\subsection{Continued fractions, moments and cumulants}
\label{subsec:moments_cumulants}

Recall from Remark~\ref{rem:classical_Sfrac} that the ordinary generating
function of the median Genocchi numbers admits a Stieltjes
continued-fraction expansion. We round out the combinatorial analysis by asking whether the
generating functions in \eqref{eq:Hid_Irr_crossnest} admit analogous
expansions and by signposting a complementary probabilistic interpretation of these objects.

For $H_n^{\mathrm{id}}(q)$, a specialization of the continued
fractions of Deb and Sokal \cite{DebSokal2024} gives
\begin{equation}
  \sum_{n\geq0} H_n^{\mathrm{id}}(q)\ssp x^n
  =
  \cfrac{1}{1-\cfrac{\alpha_1x}
  {1-\cfrac{\alpha_2x}
  {1-\cdots}}},
  \qquad
  \alpha_{2k-1}=\alpha_{2k}=k^2q^{\ssp k-1},
  \qquad k\geq1 .
  \label{eq:Hid_Sfrac}
\end{equation}
By Proposition~\ref{prop:psi_multiplicative}, the identity-bottom enumerators
and the indecomposable enumerators satisfy
\begin{equation}
  \sum_{n\geq0}H_n^{\mathrm{id}}(q)x^n
  =
  \frac{1}{
    1-\sum_{n\geq1}\mathrm{Irr}_n(q)x^n
  }.
  \label{eq:boolean_first_return}
\end{equation}
We therefore deduce the following.

\begin{Corollary}\label{cor:shifted_Irr_Sfrac}
We have
\begin{equation}
\begin{split}
 & \sum_{n\geq0}\mathrm{Irr}_{n+1}(q)\ssp x^n
  =
  \cfrac{1}{1-\cfrac{\beta_1x}
  {1-\cfrac{\beta_2x}
  {1-\cdots}}},
  \\
 & \beta_1=1, \qquad
  \beta_{2k}=\beta_{2k+1}
  =(k+1)^2q^{\ssp k},
  \qquad k\geq1.
\end{split}
  \label{eq:shifted_Irr_Sfrac}
\end{equation}
\end{Corollary}

These continued fractions of course admit a moment-theoretic interpretation. For every $q>0$, all the coefficients in
\eqref{eq:Hid_Sfrac} are strictly positive. Hence, \smash{$(H_n^{\mathrm{id}}(q)\bigr)_{n\geq0}$}
is a Stieltjes moment sequence. That is, since $H_0^{\mathrm{id}}(q)=1$, for each $q>0$,
there exists a probability measure $\mu_{H,q}$ on $[0,\infty)$
such that
\[
  H_n^{\mathrm{id}}(q)
  =
  \int_{[0,\infty)}t^n\,{\rm d}\mu_{H,q}(t).
\]
Moreover, the strict positivity of the continued-fraction coefficients implies
that every representing measure for this moment sequence has infinite support.

Likewise, \eqref{eq:shifted_Irr_Sfrac} shows that $(\mathrm{Irr}_{n+1}(q))_{n\geq0}$
is a Stieltjes moment sequence. Since $\mathrm{Irr}_{1}(q)=1$, for each $q>0$ there exists a probability measure $\mu_{I,q}$, with infinite support, such that
 \begin{equation*}
  \mathrm{Irr}_{n+1}(q) =
  \int_{[0,\infty)}t^n\,{\rm d}\mu_{I,q}(t).
\end{equation*}

Finally, the interval-decomposition generating function identity
\eqref{eq:boolean_first_return} is precisely the \emph{Boolean
moment--cumulant relation}~\cite{SpeicherWoroudi1997}. In particular,
if~$b_n^{\mu_{H,q}}$ denotes the $n$-th Boolean cumulant of~$\mu_{H,q}$, then
\[
  b_n^{\mu_{H,q}}=\mathrm{Irr}_n(q),
  \qquad n\geq1.
\]
Whether the apparent connection between Boolean probability and the algebraic and probabilistic construction at hand (arising from the central limit theorem for probability measures arising from the LLT polynomials) is coincidental or revealing of an underlying noncommutative probabilistic
structure remains to be determined.
Overall, we find that the colored inter\-lacing triangles in the two-level regime with an arbitrary number of
colors benefit from a rich combinatorial structure that points the
way to new avenues for inquiry.

\section{Computational results and observations}\label{sec:enumeration}

In this section, we present computational data for colored interlacing triangles
beyond the cases covered by Theorem~\ref{thm:Genocchi_counts} and
\cite{gaetz2023bijection}.

\subsection[Computation of new values of T\_N(n)]{Computation of new values of $\boldsymbol{T_N(n)}$}
\label{subsec:computational_approach}

We computed new values of $T_N(n)$ using a level-by-level dynamic programming
approach with GPU acceleration.
The code related to this subsection is available at the following GitHub
repository:\!
\url{https://github.com/lenis2000/colored_interlacing_triangles_enumeration/}.\!
 We enumerate interlacing colored triangles
$\boldsymbol\lambda = \bigl(\lambda^1 \prec \lambda^2 \prec \cdots \prec \lambda^N\bigr)$
by computing, for each level $k = 1, \dots, N-1$, the number of ways to extend
each configuration $\lambda^k$ to a compatible $\lambda^{k+1}$.
We enumerate only configurations in $\mathring{\mathcal{T}}_N(n)$, that is,
those with $\lambda^1=(1,2,\dots,n)$.
Moreover, in generating candidates for $\lambda^{k+1}$, we exploit the
$2^{n-1}$-symmetry from Proposition~\ref{Prop2.3}, and also
restrict the search to sequences $\lambda^{k+1}$ which start with $1$ and end
with $n$.
To obtain $T_N(n)$, we multiply the final count by $n!\ssp 2^{n-1}$.

\begin{table}[!ht]
	\centering \small
	\begin{tabular}{|c|c|c|c|c|c|c|c|}
\hline $T_N(n)\rule{0pt}{13pt}$ & $n=2$                 & $n=3$                                & $n=4$          & $n=5$       & $n=6$             & \dots & $n$           \\[2pt]
\hline $N=1$                    & 2                     & 6                                    & 24             & 120         & 720               & \dots & $n!$          \\ \hline $N=2$                    & 4                     & 48                                   & 1,344          & 72,960      & 6,796,800         & \dots &$n!\cdot H_n$ \\ \hline $N=3$                    & 8                     & 528                                  &191,232        & 257,794,560 & 1,012,737,392,640 & \dots &\textbf{?}             \\ \hline $N=4$                    & 16                    & 8,160                                & 72,099,840     & ?           & ?                 & \dots & \textbf{?}		\\ \hline $N=5$                    & 32                    & 179,520                              & 73,410,306,048 & ?           & ?                 & \dots & \textbf{?}             \\ \hline
		$N=6$            & 64                    & 5,666,304                            & ?              & ?           &?                 & \dots & \textbf{?}             \\ \hline \dots                   & \dots                & \dots                               & \dots         & \dots      & \dots            & \dots & \dots		\\ \hline $N$                      & $2^N$\rule{0pt}{12pt} &$\tfrac14\mathfrak{g}_N^{\Delta}(4)$ & \textbf{?}     &\textbf{?}  & \textbf{?}        & \dots &		\\[3pt] \hline
	\end{tabular}
	\caption{Known and unknown values of $T_N(n)$, the number of colored
	interlacing $n$-triangles with $N$ levels.
	Bold-faced question marks indicate unknown series of values.
	The columns $n=2$ and $n=3$ are~\cite{ABW2023coloured_LLT} and~\cite{gaetz2023bijection}, respectively.
	The row $N=2$ is our~Theorem~\ref{thm:Genocchi_counts}.}
	\label{tab:known_counts_of_colored_interlacing_triangles}
\end{table}

We implemented the interlacing checks on Apple Metal GPU, achieving throughput
of about~$1$ to~$2$ billion checked pairs per second on an M2 Pro chip.
For large state spaces exceeding available memory (e.g., $T_5(4)$), we process
target states in batches of $50$ million with intermediate results accumulated.
The results of the computations are summarized in
Table~\ref{tab:known_counts_of_colored_interlacing_triangles}.
All of the counts took less than a second to compute, except for $T_3(6)$, which
took about $14$ minutes, and $T_5(4)$, which took about $9$ hours.

Beyond the values reported in
Table~\ref{tab:known_counts_of_colored_interlacing_triangles} and the known cases
$N\le 2$ or $n\le 3$, the number of configurations to be checked at the final
level grows extremely rapidly, and the required computation time stretches to
multiple days or weeks.

The elements in Table~\ref{tab:known_counts_of_colored_interlacing_triangles} contain
large prime factors.
For instance, $T_5(4)$ is divisible by 331,897, and $T_3(6)$ is divisible by
457,871, and both of these divisors are primes.
This indicates that a simple product formula for the $T_N(n)$'s is highly
unlikely.
On the other hand, the $2$-adic valuations (that is, the highest powers of $2$
dividing the numbers) of the normalized counts $T_N(n)/n!$ may exhibit certain
patterns, see Table~\ref{tab:2_adic_valuations}.

\begin{table}[htpb]
	\centering
	\begin{tabular}{|c|c|c|c|c|c|}
		\hline               & $n=2$ & $n=3$ & $n=4$ & $n=5$ & $n=6$ \\
		\hline
		$N=1$ & 0     & 0     & 0     & 0     & 0 \\
		\hline
		$N=2$ & 1     & 3     & 3     & 5     & 5 \\
		\hline
		$N=3$ & 2     & 3     & 5     & 6     & 10 \\
		\hline
		$N=4$ & 3     & 4     & 8     & ?     & ? \\
		\hline
		$N=5$ & 4     & 5     & 10    & ?     & ? \\
		\hline
	\end{tabular}
	\caption{$2$-adic valuations of the normalized counts $T_N(n)/n!$.}\label{tab:2_adic_valuations}
\end{table}

\subsection[Computation of q-polynomials for two levels]{Computation of $\boldsymbol{q}$-polynomials for two levels}
\label{subsec:q_data}

We computed the $q$-counting polynomials $T_2(n;q)$ defined in
Section~\ref{sec:q_enumeration} by finding the statistic~$\psi$ for all colored
interlacing triangles of depth $N=2$.
The code and data of the statistic $\psi$ are available at the following GitHub
repository:
\begin{equation}
	\label{eq:colored_interlacing_triangles_q_enumeration_github}
	\text{\url{https://github.com/lenis2000/colored_interlacing_triangles_q_enumeration}}
\end{equation}
For $N=2$, a colored interlacing $n$-triangle consists of a permutation
$\lambda^1 \in S_n$ (level~$1$) and a sequence $\lambda^2$ of length $2n$ where
each color appears exactly twice (level~$2$).
The interlacing condition requires that for each color $c$, its position in
$\lambda^1$ lies strictly between its two positions in $\lambda^2$.

The algorithm explores the same symmetries as in
Section~\ref{subsec:computational_approach} to reduce the search space.
For each ``canonical'' triangle from $\mathring{\mathcal{T}}_2(n)$ (with
identity permutation at level~$1$), we compute the $\psi$-statistic for all $n!$
color permutations in the whole triangle.
Namely, for each valid interlacing triangle $\bigl(\lambda^1, \lambda^2\bigr)$, we compute
$\psi\bigl(\lambda^1, \lambda^2\bigr)$ using bitmask operations: we process positions
right-to-left, maintaining a bitmask of ``active'' colors (those that have
entered from the right but not yet exited upward).
The contribution to the $q$-power is computed via hardware-accelerated
population count (\texttt{popcount}) instructions.
The available data on GitHub also accounts for the distribution of
$q$-statistics over all $n!$ permutations for each canonical triangle.

The implementation uses \texttt{C++} template metaprogramming to specialize the
inner loop for each $n$, enabling complete loop unrolling and instruction-level
parallelism (processing 4 permutations simultaneously).
Combined with OpenMP parallelization, we achieve the following benchmarks on an
Apple M2 Pro (6 performance cores):
\begin{center}
	{\renewcommand{\arraystretch}{1.15}%
	\setlength{\tabcolsep}{8pt}%
	\begin{tabular}{c|r|r}
\hline $n$       & \multicolumn{1}{c|}{Canonical triangles} &\multicolumn{1}{c}{Time} \\ \hline
		6 & 295                                      & $< 1$ second \\
7                & 3,098                                    & $< 1$ second \\
8                & 42,271                                   & $\approx 28$ seconds \\
9                & 726,734                                  & $\approx 1$ hour \\ \hline
	\end{tabular}}
\end{center}

Writing $T_2(n;q) = 2^{n-1} P_n(q)$, the polynomials $P_n(q)$ are of degree
$\binom{n}{2}$ and are palindromic (indicated by the symbol $\circlearrowleft$
below).
Here are their explicit forms for $n$ up to $9$:
\begin{align*}
&P_1(q)= 1, \\
&P_2(q)= 1+q, \\
&P_3(q)= 1+5q+5q^2+q^3, \\
&P_4(q)= 1+10q+47q^2+52q^3+47q^4+10q^5+q^6, \\
&P_5(q)= 1+15q+132q^2+527q^3+1019q^4+1172q^5+\circlearrowleft
	, \\
&P_6(q)= 1+20q+245q^2+1825q^3+7295q^4+19534q^5+34465q^6
	+42815q^7+\circlearrowleft, \\
&P_7(q)= 1+25q+383q^2+3977q^3+26645q^4+115165q^5+365346q^6 \\
&\hphantom{P_7(q)= }{}
+878276q^7+1563964q^8+2226948q^9+2626230q^{10}+\circlearrowleft, \\
&P_8(q)= 1+30q+546q^2+7018q^3+64622q^4+411692q^5+1914780q^6 \\
       &\hphantom{P_8(q)=}{}
        +6889907q^7+19865655q^8+46208719q^9+88274748q^{10} \\
       &\hphantom{P_8(q)=}{} +141139717q^{11}+193232778q^{12}+231337829q^{13}+245670636q^{14} +\circlearrowleft,\\
&P_9(q)=
	1 + 35q + 734q^2 + 11064q^3 + 125319q^4 + 1059757q^5 + 6649287q^6 + 32573212q^7\\
       &\hphantom{P_9(q)=}{}+ 129428316q^8 + 424672873q^9 + 1174423848q^{10} \\
        &
 \hphantom{P_9(q)=}{}
        + 2770263242q^{11} +
	5640036376q^{12}  + 9998171117q^{13} + 15583534941q^{14}
\\       &\hphantom{P_9(q)=}{}
	+21645762974q^{15}   + 27163602028q^{16} + 31055190622q^{17} +
	32466222428q^{18} + \circlearrowleft.
\end{align*}

\begin{Conjecture}
	\label{conj:log_concavity} The coefficients of each polynomial $P_n(q)$
	$($for each fixed $n)$ are
	log-concave: if $P_n(q) = \sum_{k} a_k(n) q^k$, then we have
	$a_k(n)^2 \ge a_{k-1}(n) a_{k+1}(n)$ for all $n$, $k$.
	We have verified this for $n \le 9$.
\end{Conjecture}

\begin{Remark}
	\label{rem:not_real_rooted}
	The polynomials $P_n(q)$ are not real-rooted: for $n\le9$ we find that
	$P_n(q)$ has non-real roots once $n\ge4$ (already all six roots of $P_4(q)$
	are non-real) and roots of positive real part once $n\ge6$.
	In particular, Conjecture~\ref{conj:log_concavity} cannot be deduced from
	real-rootedness of $P_n(q)$.
\end{Remark}

\subsection[Coefficients of P\_n(q)]{Coefficients of $\boldsymbol{P_n(q)}$}
\label{subsec:conjectures_on_coefficients}

We implemented a different strategy to access low-degree coefficients of
$P_n(q)$ beyond $n=9$, where the full polynomial computation becomes too slow.
The code for this approach is available at the repository
\eqref{eq:colored_interlacing_triangles_q_enumeration_github}.

Rather than generating all Dumont derangements for a given $n$ (corresponding to
interlacing triangles of depth $2$ with identity permutation at the bottom row)
and then computing the $\psi$-statistics for all $n!$ permutations of colors, we
\emph{invert the loops}: the outer loop iterates over all $n!$ bottom row
permutations $\lambda^1$, while the inner loop generates Dumont derangements
on-the-fly via backtracking.
The key observation is that the $\psi$-statistic accumulates during the
backtracking process.
When computing only coefficients of $q^k$ for $k<K_{\max}$ (with $K_{\max}$
small), we prune any branch where the partial $\psi$-value reaches $K_{\max}$.
This pruning eliminates a large portion of the Dumont search space.

A further optimization exploits the relationship between the inversion number of
the bottom row permutation $\lambda^1$ and the $\psi$-statistic.
Empirically, permutations with large inversion numbers tend to produce only high
$\psi$-values across all Dumont derangements.
By restricting to bottom row permutations with
$\operatorname{inv}\bigl(\lambda^1\bigr) \leq I_{\max}$ for a suitable threshold
$I_{\max}$, we reduce the outer loop dramatically while still capturing all
contributions to low-degree coefficients.
Combined, these optimizations make it feasible to compute low-degree
coefficients up to $n = 15$.

We obtain the following initial pieces of further polynomials $P_n(q)$:
\begin{align*}
	&P_{10}(q)= 1 + 40q + 947q^2 + 16240q^3 + 214297q^4 + 2207081q^5 +\cdots,	\\
	&P_{11}(q)= 1+45q + 1185q^2 + 22671q^3 + 338129q^4 +4033256 q^5+\cdots,	\\
	&P_{12}(q)= 1 + 50q + 1448q^2 + 30482q^3 + 504040q^4 +6762968 q^5 +\cdots,	\\
	&P_{13}(q)= 1 + 55q + 1736q^2 + 39798q^3 + 719880q^4 +10663371q^5 +\cdots,	\\
	&P_{14}(q)= 1 + 60q + 2049q^2 + 50744q^3 + 994124q^4 +16045700 q^5 +\cdots,	\\
	&P_{15}(q)= 1 + 65q + 2387q^2 + 63445q^3 + 1335872q^4 + 23268315 q^5+
	\cdots.
\end{align*}
Based on this data, we conjecture the following polynomial behavior of all the
remaining coefficients of $P_n(q)$, with concrete polynomials for the next four
(recall that the linear coefficient is $a_1(n)=5(n-2)$, see
Proposition~\ref{prop:linear_coefficient}).

\begin{Conjecture}
	\label{conj:coeff_polynomials} For each fixed $k \geq 0$, and all
	$n\ge n_0(k)$, the coefficient $a_k(n)$ of $q^k$ in $P_n(q)$ is a polynomial
	in $n$ of degree $k$.
	Specifically, we conjecture the following polynomials for the next four
	coefficients:
	\begin{gather}
		\label{eq:conj_coeff_polynomials}
		\begin{split}
			&a_2(n)=\frac{25n^2 - 49n - 116}{2},\qquad n \geq 5,\\
			&a_3(n)=\frac{125n^3 + 15n^2 - 3104n + 1980}{6},\qquad  n \geq 7,\\
			&a_4(n)=\frac{625 n^4 + 2650 n^3 - 36877 n^2 - 31390 n + 244728}{24},\qquad  n \geq 9,\\
			&a_5(n)=\frac{3125 n^5 + 32500 n^4 - 290925 n^3 - 1585240 n^2 +
			7120060 n + 5588400}{120},\qquad  n \geq 11.
		\end{split}\hspace{-10mm}
	\end{gather}
\end{Conjecture}

\begin{Conjecture}[leading coefficient]
	\label{conj:leading_coefficient} As polynomials in $n$, the expressions
	$k!\ssp a_k(n)$ have integer coefficients for all $k \geq 0$, and their
	leading coefficient is $5^k$.
\end{Conjecture}
\begin{Remark}
	\label{rmk:log_concavity} The polynomials \eqref{eq:conj_coeff_polynomials}
	(together with $a_1(n)=5n-10$) satisfy log-concavity
	$a_k(n)^2 \ge a_{k-1}(n) a_{k+1}(n)$ for $k=2,3,4$ and all $n>0$.
	This supports our Conjecture~\ref{conj:log_concavity}.
\end{Remark}

\begin{Remark}
	\label{rem:indecomposable_triangles}
	The $q$-statistic also
	allows us to deform the enumeration of indecomposable
	triangles.
	The first few of the polynomials $\mathrm{Irr}_n(q)$ from
	Proposition~\ref{prop:psi_multiplicative} are
	\begin{align*}
			&\mathrm{Irr}_1(q)=\mathrm{Irr}_2(q)=1,\\
			&\mathrm{Irr}_3(q)=1+4q,\\
			&\mathrm{Irr}_4(q)=1+8q+32q^2,\\
			&\mathrm{Irr}_5(q)=1+12q+80q^2+256q^3+144q^4,\\
			&\mathrm{Irr}_6(q)=1+16q+144q^2+768q^3+2336q^4+2304q^5+2592q^6,\\
			&\mathrm{Irr}_7(q)=1+20q+224q^2+1600q^3+7600q^4+22144q^5\\
			&\hphantom{\mathrm{Irr}_7(q)=}{} +32832q^6+46656q^7+46656q^8+20736q^9,\\
			&\mathrm{Irr}_8(q)=1+24q+320q^2+2816q^3+17472q^4+75904q^5+221792q^6\\
			&\hphantom{\mathrm{Irr}_8(q)=}{} +408960q^7+715392q^8+974592q^9+1171584q^{10}+746496q^{11}
			+663552q^{12}.
	\end{align*}
	The polynomials $\mathrm{Irr}_n(q)$ were extracted from the
	$H_n^{\mathrm{id}}(q)$'s via the generating function
	identity~\eqref{eq:first_return_identity_bottom}. The
	$H_n^{\mathrm{id}}(q)$ are part of the enumeration data described in
	Section~\ref{subsec:q_data}.
	As a consistency check, the evaluations $\mathrm{Irr}_n(1)$ reproduce
	the terms of
	the sequence $h^\flat$
	(Corollary~\ref{cor:indecomposable_count} and
	\cite[A362111]{OEIS}) for all $n\le 8$.

	The companion polynomials $\widetilde{\mathrm{Irr}}_n(q)$ of
	Definition~\ref{def:indecomposable_general}, which enumerate indecomposable
	triangles with \emph{arbitrary} bottom rows, can be extracted from the
	polynomials $P_n(q)$ of Section~\ref{subsec:q_data} via the $q$-exponential
	identity \eqref{eq:q_exponential_first_return}:
	\begin{align*}
		&\widetilde{\mathrm{Irr}}_1(q)=1,\\
		&\widetilde{\mathrm{Irr}}_2(q)=1+q,\\
		&\widetilde{\mathrm{Irr}}_3(q)=1+14q+14q^2+q^3,\\
		&\widetilde{\mathrm{Irr}}_4(q)=1+35q+293q^2+326q^3+293q^4
		+35q^5+q^6,\\
		&\widetilde{\mathrm{Irr}}_5(q)=1+56q+1157q^2+6627q^3+13728q^4
		+16022q^5+\circlearrowleft,\\
		&\widetilde{\mathrm{Irr}}_6(q)=1+77q+2526q^2+35189q^3+176257q^4\\
&\hphantom{\widetilde{\mathrm{Irr}}_6(q)=}{}
		+528839q^5+972546q^6+1222525q^7+\circlearrowleft.
	\end{align*}
	In agreement with Corollary~\ref{cor:q_analog_full_hflat}, the evaluations
	$\widetilde{\mathrm{Irr}}_n(1)=n!\ssp h^\flat_n$ hold for all
	$n\le 9$, with
\[
h^\flat_7=178469 , \qquad h^\flat_8=4998905 , \qquad \text{and}\qquad h^\flat_9=174914077.
\]
Since the color complement involution of Lemma~\ref{lem:color_complement} preserves
indecomposability, the polynomials
\smash{$\widetilde{\operatorname{Irr}}_n(q)$} are palindromic of degree
$\binom{n}{2}$, similarly to $T_2(n;q)$ (Proposition~\ref{prop:q_polynomial_properties}).
\end{Remark}

We can make the lower-bound direction of Conjecture~\ref{conj:leading_coefficient}
rigorous in the following asymptotic sense:{\samepage
\[
\liminf_{n\to\infty}k!\ssp a_k(n)/n^{k}\ge 5^{k}
\] for every fixed $k$.
Throughout we use the combinatorial description \eqref{eq:psi_combinatorial} of
$\psi$, and we work modulo the $2^{n-1}$ top-row involutions
(Proposition~\ref{Prop2.3}). For $1\le i\le n-1$, the $i$-th involution
transposes the two top entries $\lambda[i]^2_2$ and $\lambda[i+1]^2_1$ that are
immediate neighbors, with no bottom between them, in the linear order~\eqref{eq:interlacing_order}. By Proposition~\ref{Prop2.3}, these two
entries are always distinct, the transposition preserves validity and $\psi$,
and the $n-1$ involutions are independent, so every orbit has exactly $2^{n-1}$
elements and $a_k(n)$ equals the number of orbits with $\psi=k$. We fix the
orbit representative obtained by sorting each such neighboring pair
$\bigl(\lambda[i]^2_2,\lambda[i+1]^2_1\bigr)$ into increasing order, and call it
\emph{normalized}.

}

We now record five explicit defects at a block-start $p$. Their bottom rows and
the ordered pairs of top dots of successive triangles are listed with colors
written relative to~$p$ (every other triangle is canonical).
\begin{equation}
	\label{eq:five_defects}
	\renewcommand{\arraystretch}{1.2}
	\begin{array}{l@{\qquad}l@{\qquad}l}
		\text{type} & \text{bottom row on the block} & \text{top dots $\bigl(t^{\mathrm L},t^{\mathrm R}\bigr)$ on the block}\\[2pt]\hline
		S\ (w{=}2) & (p{+}1,p) & (p{+}1,p{+}1),(p,p)\\
		A\ (w{=}3) & (p,p{+}1,p{+}2) & (p,p{+}1),(p{+}2,p),(p{+}1,p{+}2)\\
		B\ (w{=}3) & (p,p{+}2,p{+}1) & (p,p{+}2),(p{+}1,p),(p{+}2,p{+}1)\\
		C\ (w{=}3) & (p{+}1,p,p{+}2) & (p{+}1,p),(p{+}2,p{+}1),(p,p{+}2)\\
		D\ (w{=}4) & (p,p{+}2,p{+}1,p{+}3) & (p,p{+}1),(p{+}2,p),(p{+}3,p{+}2),(p{+}1,p{+}3)
	\end{array}
\end{equation}
A direct check from \eqref{eq:psi_combinatorial} shows that each of
$S$, $A$, $B$, $C$, $D$ is a valid, color-balanced defect on its block, with $\psi=1$, and
that it differs from $\lambda^{\mathrm{can}}$ at \emph{every} triangle of its
block. They refine the local analysis in the proof of
Proposition~\ref{prop:linear_coefficient}: type~$A$ has $\sigma=\mathrm{id}$, while each of
$S$, $B$, $C$, $D$ uses a single adjacent transposition.

\begin{Proposition}
	\label{prop:leading_lower_bound}
	For all $k\ge 1$ and all $n\ge 1$, we have
	the following inequality $($here and below $\binom{m}{k}=0$ whenever $m<k$,
	so the bound is nontrivial only for $n\ge 5k-1)$:
	\begin{equation}
		\label{eq:leading_lower_bound}
		a_k(n)\ \ge\ 5^{k}\binom{n+1-4k}{k}
		 = \frac{5^{k}}{k!} n^{k}+O \bigl(n^{k-1}\bigr),
		\qquad n\to\infty.
	\end{equation}
	In particular, $\liminf_{n\to\infty}k!\ssp a_k(n)/n^{k}\ge 5^{k}$ for each
	fixed $k$. Thus, if $a_k(n)$ is eventually a~polynomial in $n$ of degree $k$
	as in Conjecture~{\rm \ref{conj:coeff_polynomials}}, then the leading coefficient of
	$k!\ssp a_k(n)$ is at least~$5^{k}$, which is the lower-bound half of
	Conjecture~{\rm \ref{conj:leading_coefficient}}.
\end{Proposition}

\begin{proof}
	If $n<5k-1$, then the right-hand side of \eqref{eq:leading_lower_bound}
	vanishes and there is nothing to prove. Assume $n\ge 5k-1$.
	Choose block starts $p_1<p_2<\cdots<p_k$ in $\{1,\dots,n-3\}$ with
	$p_{r+1}-p_r\ge 5$ for $1\le r<k$, and at each $p_r$ one of the five defects
	of \eqref{eq:five_defects}, with block $B_r$ (of width $w_r\le 4$, starting
	at $p_r$). Since every defect has width at most $4$ and $p_{r+1}-p_r\ge5$,
	the blocks $B_1,\dots,B_k$ are pairwise disjoint and separated by at least one
	canonical triangle; and $p_k\le n-3$ guarantees $B_k\subseteq\{1,\dots,n\}$.
	Let $\Lambda$ be the configuration assembled from these defects. By
	Lemma~\ref{lem:defect_additivity}, $\Lambda$ is valid and
	$\psi(\Lambda)=\sum_{r}\psi(\Lambda_r)=k$.

	\emph{Distinctness.} We show the orbit of $\Lambda$ determines the data
	$((p_1,\tau_1),\dots,(p_k,\tau_k))$, where $\tau_r$ is the chosen
	type. The bottom row $\sigma$ is invariant under the top-row involutions, and
	the normalized top is the canonical representative of the orbit. Thus the pair
	$ (\sigma,\widehat{\nu} )$, with $\widehat\nu$ the normalized top, is
	an orbit invariant. Each of the five defects differs from
	$\lambda^{\mathrm{can}}$, in this invariant, at every triangle of its block,
	and at no triangle outside it. So the set of triangles where $\Lambda$ differs
	from $\lambda^{\mathrm{can}}$ is exactly $\bigcup_r B_r$. As the blocks are
	separated by canonical triangles, its maximal runs of consecutive triangles
	are precisely the $B_r$, which recovers each $p_r$ (the run minimum) and each
	width $w_r$. Finally the type is read off: widths $2$ and $4$ each occur for a
	single type ($S$, $D$), while the three width-$3$ types are distinguished by
	their bottom rows $(p,p{+}1,p{+}2)$, $(p,p{+}2,p{+}1)$, $(p{+}1,p,p{+}2)$ for
	$A$, $B$, $C$. Hence the assignment data $\mapsto\Lambda$ is injective, and its
	images lie in pairwise distinct orbits, each with $\psi=k$.

	\emph{Count.} The number of choices of $p_1<\cdots<p_k$ in $\{1,\dots,n-3\}$
	with consecutive gaps $\ge5$ equals $\binom{(n-3)-4(k-1)}{k}=\binom{n+1-4k}{k}$,
	and each carries $5^k$ independent type choices. Therefore,
	$a_k(n)\ge 5^k\binom{n+1-4k}{k}$, and
	$\binom{n+1-4k}{k}=\frac{1}{k!}\prod_{j=0}^{k-1}(n+1-4k-j)=\frac{n^k}{k!}+O(n^{k-1})$
	gives~\eqref{eq:leading_lower_bound}.
\end{proof}

\begin{Remark}
	\label{rem:leading_lower_bound}
	For $k=1$,
	Proposition~\ref{prop:leading_lower_bound}
	gives $a_1(n)\ge 5(n-3)$, slightly
	weaker than the exact formula $a_1(n)=5(n-2)$ of Proposition~\ref{prop:linear_coefficient}.

	The matching upper bound $a_k(n)\le\frac{5^k}{k!}n^k+O\bigl(n^{k-1}\bigr)$
	(consistent with all computed data) would, combined with
	Proposition~\ref{prop:leading_lower_bound}, give
	$k!\ssp a_k(n)=5^{k}n^{k}+O\bigl(n^{k-1}\bigr)$, the leading-order behavior
	conjectured in Conjecture~\ref{conj:leading_coefficient}. It amounts to showing that
	$\psi=k$ forces all but $O\bigl(n^{k-1}\bigr)$ configurations to be a union of $k$
	such independent unit defects.
	We do not pursue this bound here.
\end{Remark}

\subsection{Linear cumulant phenomenon}
\label{subsec:linear_cumulant_phenomenon}

If we apply the classical ``moment-cumulant'' transformation (corresponding to
taking the logarithm of the exponential generating function) to the sequence of
the normalized coefficients $m_k(n)\coloneqq k!\ssp a_k(n)$ coming from
\eqref{eq:conj_coeff_polynomials}, then we obtain the following linear functions
of $n$:
\begin{equation*}
	\begin{split}
		&m_1=5n-10,\\
		&m_2-m_1^2=51n-216,\\
		&m_3-3m_2m_1+2m_1^3
=166n-3500,\\
		&m_4 -4m_3m_1-3m_2^2+12m_2m_1^2-6m_1^4
=-28854n+84360,\\
		&m_5 -5m_4m_1-10m_3m_2+ 20m_3m_1^2+30m_2^2m_1-60m_2m_1^3+24m_1^5\\
&\qquad = -
		1258080 n +10684800 .
	\end{split}
\end{equation*}

\begin{Conjecture}
	\label{conj:linear_cumulants}
	This pattern of linear cumulants continues for all $k\ge1$.
\end{Conjecture}

The linear cumulant phenomenon aligns with the local independence picture
discussed after Conjectures~\ref{conj:coeff_polynomials} and~\ref{conj:leading_coefficient}: if the
$\psi$-statistic arises from choosing independent local defects at disjoint
interfaces, then the distribution of $\psi$ should be approximately a sum of
independent random variables, one per interface.
For such sums, cumulants grow linearly in the number of summands.

\subsection{Moment sequences, revisited}\label{subsec:moment_sequences}

By convention, let us add the initial value $T_2(0;q)=1$.
Let us consider the question whether the sequence $T_2(n;q)=2^{n-1}P_n(q)$,
$n\ge0$, is a Stieltjes moment sequence in $n$ for each fixed $q\in[0,1]$, that
is, there exists a nonnegative Borel measure $\mu_q$ on $[0,+\infty)$ such that
\begin{equation*}
	T_2(n;q)  =  \int_{0}^{+\infty} x^n \, {\rm d}\mu_q(x), \qquad n=0,1,2,\dots.
\end{equation*}
We know that this is the case for $q=0$ and $q=1$ (for the evaluation of
$T_2(n;0)$, see the proof of Proposition~\ref{prop:linear_coefficient}):
\begin{equation*}
	T_2(n;0) =\begin{cases}
		1,       & n=0, \\
		2^{n-1}, & n\ge1,
	\end{cases}
	\qquad T_2(n;1) = n!
	\cdot H_n.
\end{equation*}
Indeed, for $q=0$, we have $\mu_0=\frac{1}{2}\delta_0+\frac{1}{2}\delta_2$, and
for $q=1$, the sequences $\{n!\}$ and $\{H_n\}$ are both Stieltjes moment
sequences, so their product is also a Stieltjes moment sequence.
For $\{H_n\}$, the Stieltjes moment property follows from the classical Stieltjes
continued fraction for the median Genocchi numbers, with coefficients
$\alpha_{2k-1}=\alpha_{2k}=k^2$; see \cite{DebSokal2024}.

For other values of $q\in(0,1)$, we employ the classical test based on two
Hankel determinants, $\det[ T_2(i+j-2;q) ]_{i,j=1}^k$ and
$\det[ T_2(i+j-1;q) ]_{i,j=1}^k$, for $k$ up to $5$.
We find that both determinants may be negative for sufficiently small $q>0$, for
instance:
\begin{equation*}
	\det[ T_2(i+j-2;q) ]_{i,j=1}^5<0,\qquad 0<q<c_1, \qquad c_1\approx
	0.08462\dots,
\end{equation*}
and
\begin{equation*}
	\det[ T_2(i+j-1;q) ]_{i,j=1}^3<0,\qquad 0<q<c_2, \qquad c_2\approx
	0.04641\dots.
\end{equation*}
The values $c_1$ and $c_2$ are the smallest positive roots of the respective
determinants viewed as polynomials in $q$.
Thus, for sufficiently small $q>0$, the sequence $\{T_2(n;q)\}_{n\ge0}$ is not a~Stieltjes moment sequence in~$n$.
Since the property holds at the endpoints $q=0$ and $q=1$, the moment sequence
property undergoes a transition at $q=0^+$, ruling out a uniform representation
of~$T_2(n;q)$ as moments of a family of nonnegative Borel measures varying in
$q$ over $[0,1]$.

\section{Sampling depth-2 colored interlacing triangles}\label{sec:sampling}

Here, we describe a simple Markov chain which preserves the $q$-weighted
distribution on depth\nobreakdash-$2$ colored interlacing triangles.
That is, each triangle $\boldsymbol\lambda\in\mathcal{T}_2(n)$ is sampled with
probability proportional to $q^{\psi(\boldsymbol\lambda)}$; throughout this
section, we assume $q>0$.

The code is available at the GitHub repository
\[
	\text{\url{https://github.com/lenis2000/colored_interlacing_triangles_q_sampling}}
\]
Define a Markov chain on the state space $\mathcal{T}_2(n)$ that uses two types
of local moves.

\begin{Definition}[level-$2$ swap]
	\label{def:level2_swap} Given a position $p \in \{1, \dots, 2n-1\}$, the
	\emph{level-$2$ swap at position $p$} exchanges $\lambda^2_p$ and
	$\lambda^2_{p+1}$.
	The move is \emph{valid} if the resulting configuration still satisfies the
	interlacing constraint.
\end{Definition}

\begin{Definition}[level-$1$ swap with reconciliation]
	\label{def:level1_swap} Given a position $i \in \{1, \dots, n-1\}$, let
	$a = \lambda^1_i$ and $b = \lambda^1_{i+1}$.
	The \emph{level-$1$ swap at position $i$} exchanges $a$ and $b$ in
	$\lambda^1$, then applies a deterministic \emph{reconciliation} to
	$\lambda^2$ to restore interlacing.

	Let the \emph{between region} consist of positions $\lambda^2_{2i}$ and
	$\lambda^2_{2i+1}$ (the level-$2$ entries that lie between the $i$-th and
	$(i+1)$-st positions in the linear order \eqref{eq:interlacing_order}).
	The reconciliation proceeds as follows:
	\begin{enumerate}[\rm (i)]\itemsep=0pt
		\item If neither $a$ nor $b$ appears in the between region, no change to $\lambda^2$ is needed.
		\item If only $a$ appears in the between region, replace it with $b$,
		then find a copy of $b$ to the right of the between region, and replace it
		with $a$.
		\item If only $b$ appears in the between region, then act symmetrically to (ii):
		replace $b$ with $a$ in the between region, then replace a copy of $a$
		to the left of the between region with $b$.
		\item If both $a$ and $b$ appear in the between region, find $a$ to the left of the between region
		and $b$ to the right, and swap them.
	\end{enumerate}
	This move is always valid: one readily sees that interlacing is preserved.
\end{Definition}

\begin{Proposition}[connectedness of the state space]
	\label{prop:connected_state_space} The two types of swaps from
	Definitions~{\rm \ref{def:level2_swap}} and~{\rm \ref{def:level1_swap}} connect all states in
	$\mathcal{T}_2(n)$.
\end{Proposition}
\begin{proof}
	We show that any triangle can be transformed to the identity one with
	$\lambda^1 = (1, 2, \dots, n)$ and
	$\lambda^2 = (1, 1, 2, 2, \dots, n, n)$.
	First, level-1 swaps can sort $\lambda^1$ to the identity permutation.
	The reconciliation ensures that $\lambda^2$ is adjusted to maintain
	interlacing.

	Second, once $\lambda^1 = (1, 2, \dots, n)$, the level-$2$ swaps can
	rearrange $\lambda^2$ within the constraints imposed by interlacing.
	Indeed, one copy of $1$ in the second level is fixed at position $1$.
	Let us move the other copy of $1$ to position $2$ using level-$2$ swaps
	which move $1$ leftward.
	The only obstruction to the leftward movement of $1$ is a configuration of
	the form $(\dots b a|a|1c\dots)$, which prevents swapping $1$ and $a$ on
	the second level.
	However, since $a\ne b$, we can first swap~$a$ and~$b$ on the second level,
	and then swap~$1$ and~$b$, resulting in the configuration
	$(\dots a1|a|bc\dots)$.
	This allows us to proceed with moving $1$ leftward.
	After the other copy of $1$ is at position $2$, we can remove the leftmost
	triangle consisting of all $1$'s, and repeat the same procedure for
	$2,3,\dots,n$.
\end{proof}

We can thus apply the \emph{Metropolis--Hastings} (also called the \emph{Markov
chain Monte Carlo}) \emph{algorithm} with our swaps.
In detail, at each step, uniformly choose one of the $3n-2$ possible swap
positions at both levels, attempt the corresponding move (in particular, check
the validity of the level-2 swap), and accept or reject according to the
Metropolis--Hastings rule with acceptance probability
\begin{equation*}
	\alpha(\boldsymbol\lambda \to \boldsymbol\lambda') = \min\bigl(1,
	q^{\psi(\boldsymbol\lambda') - \psi(\boldsymbol\lambda)}\bigr).
\end{equation*}
It is well-known
\cite{LevinPeres2017}
that this procedure defines a Markov chain with stationary distribution
proportional to $q^{\psi(\boldsymbol\lambda)}$.

The results of sampling are illustrated in
Figures~\ref{fig:sampled_triangles}, \ref{fig:sampling_heatmaps} and \ref{fig:psi_histogram}.
The effect of the $q$-weighting is visible both in the sampled triangles and in
the heatmaps.
The latter resemble the permuton limit of the Mallows measure on permutations
\cite{Starr2009}.
Figure~\ref{fig:psi_histogram} shows the empirical distribution of the
$\psi$-statistic, which appears approximately Gaussian.
This is not entirely unexpected, as asymptotic normality is known for various
statistics on random permutations. For example, the number of crossings of a
uniformly random permutation is asymptotically Gaussian, by the
equidistribution of crossings with superfluous 1's in permutation tableaux
\cite{Corteel2007crossings} and the central limit theorem for the latter
\cite{HitczenkoJanson2010}.
The observations in this section are based on numerical simulation. We present
them as computational exploration, and regard proving any of them as an
interesting direction for future work.

\begin{figure}[htpb]
	\centering
	\resizebox{\textwidth}{!}{%
	\begin{tikzpicture}[y=4.5cm]
		\foreach \c/\lpos/\ltpos/\rtpos in {%
		1/1/1/3, 2/3/4/9, 3/2/2/6, 4/5/5/12, 5/7/8/15, 6/8/10/16, 7/9/14/20, 8/4/7/19,%
		9/10/13/23, 10/6/11/17, 11/14/21/32, 12/11/18/25, 13/13/24/31, 14/15/28/34, 15/12/22/35, 16/19/27/39,%
		17/18/33/36, 18/17/30/37, 19/20/26/40, 20/16/29/42, 21/21/41/44, 22/25/38/57, 23/26/47/52, 24/24/46/48,%
		25/22/43/50, 26/23/45/53, 27/27/49/55, 28/28/54/56, 29/30/59/63, 30/33/62/69, 31/29/51/60, 32/32/61/64,%
		33/34/65/77, 34/36/66/72, 35/37/68/75, 36/35/67/70, 37/31/58/73, 38/40/76/81, 39/39/74/79, 40/38/71/89,%
		41/42/82/85, 42/43/80/90, 43/41/78/84, 44/44/83/88, 45/45/86/93, 46/46/91/92, 47/48/94/96, 48/47/87/97,%
		49/49/95/98, 50/50/99/100%
		}{%
		\pgfmathsetmacro{\hue}{\c/50}%
		\definecolor{col}{hsb}{\hue,0.8,0.8}%
		\draw[col, line width=2pt] ({\lpos-1},0) -- ({\ltpos*0.5-0.5},1);%
		\draw[col, line width=2pt] ({\lpos-1},0) -- ({\rtpos*0.5-0.5},1);%
		\fill[col] ({\lpos-1},0) circle (7pt);%
		\fill[col] ({\ltpos*0.5-0.5},1) circle (7pt);%
		\fill[col] ({\rtpos*0.5-0.5},1) circle (7pt);%
		}
	\end{tikzpicture}}

	\vspace{1em}

	\resizebox{\textwidth}{!}{%
	\begin{tikzpicture}[y=4.5cm]
		\foreach \c/\lpos/\ltpos/\rtpos in {%
		1/16/27/49, 2/3/3/37, 3/10/15/53, 4/35/31/82, 5/25/17/73, 6/5/9/67, 7/29/36/64, 8/33/51/81,%
		9/11/7/40, 10/44/41/97, 11/15/28/33, 12/39/30/80, 13/1/1/44, 14/13/16/60, 15/45/19/91, 16/38/63/96,%
		17/19/23/65, 18/34/50/92, 19/48/54/99, 20/2/2/76, 21/50/57/100, 22/4/4/26, 23/14/5/39, 24/47/34/94,%
		25/12/18/43, 26/31/14/69, 27/6/11/71, 28/28/21/77, 29/24/20/74, 30/17/12/42, 31/42/70/84, 32/37/52/89,%
		33/20/38/56, 34/23/13/66, 35/46/62/95, 36/9/8/25, 37/26/45/72, 38/40/46/83, 39/41/24/87, 40/49/93/98,%
		41/36/58/75, 42/7/6/90, 43/8/10/79, 44/18/32/47, 45/43/61/86, 46/21/22/59, 47/32/48/85, 48/22/35/78,%
		49/27/29/68, 50/30/55/88%
		}{%
		\pgfmathsetmacro{\hue}{\c/50}%
		\definecolor{col}{hsb}{\hue,0.8,0.8}%
		\draw[col, line width=2pt] ({\lpos-1},0) -- ({\ltpos*0.5-0.5},1);%
		\draw[col, line width=2pt] ({\lpos-1},0) -- ({\rtpos*0.5-0.5},1);%
		\fill[col] ({\lpos-1},0) circle (7pt);%
		\fill[col] ({\ltpos*0.5-0.5},1) circle (7pt);%
		\fill[col] ({\rtpos*0.5-0.5},1) circle (7pt);%
		}
	\end{tikzpicture}}\vspace{-1mm}

	\caption{Sampled depth-$2$ colored interlacing triangles with $n=50$ colors,
	ordered as a rainbow from color $1$ (red) to color $50$ (violet).
	The two triangles were obtained after MCMC runs of $10^7$ steps each.
	Top: $q=0.2$ (here the $\psi$-statistic is $55$).
	Bottom: $q=0.98$ (here $\psi=541$).
	Lines connect each color's occurrences across levels, and indicate the
	interlacing properties.} \label{fig:sampled_triangles}\vspace{-2mm}
\end{figure}

\begin{figure}[!ht]
	\centering
	\includegraphics[width=0.39\textwidth]{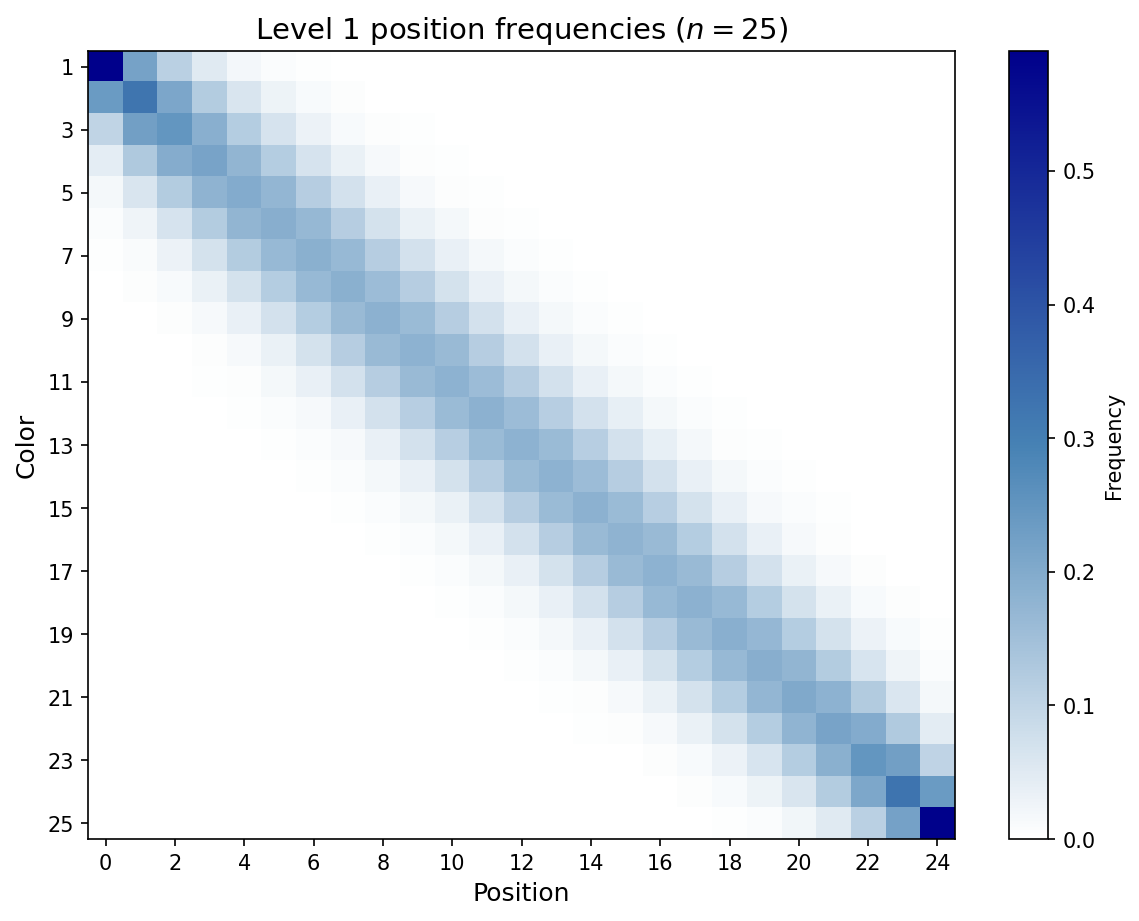}
	\hfill
	\includegraphics[width=0.57\textwidth]{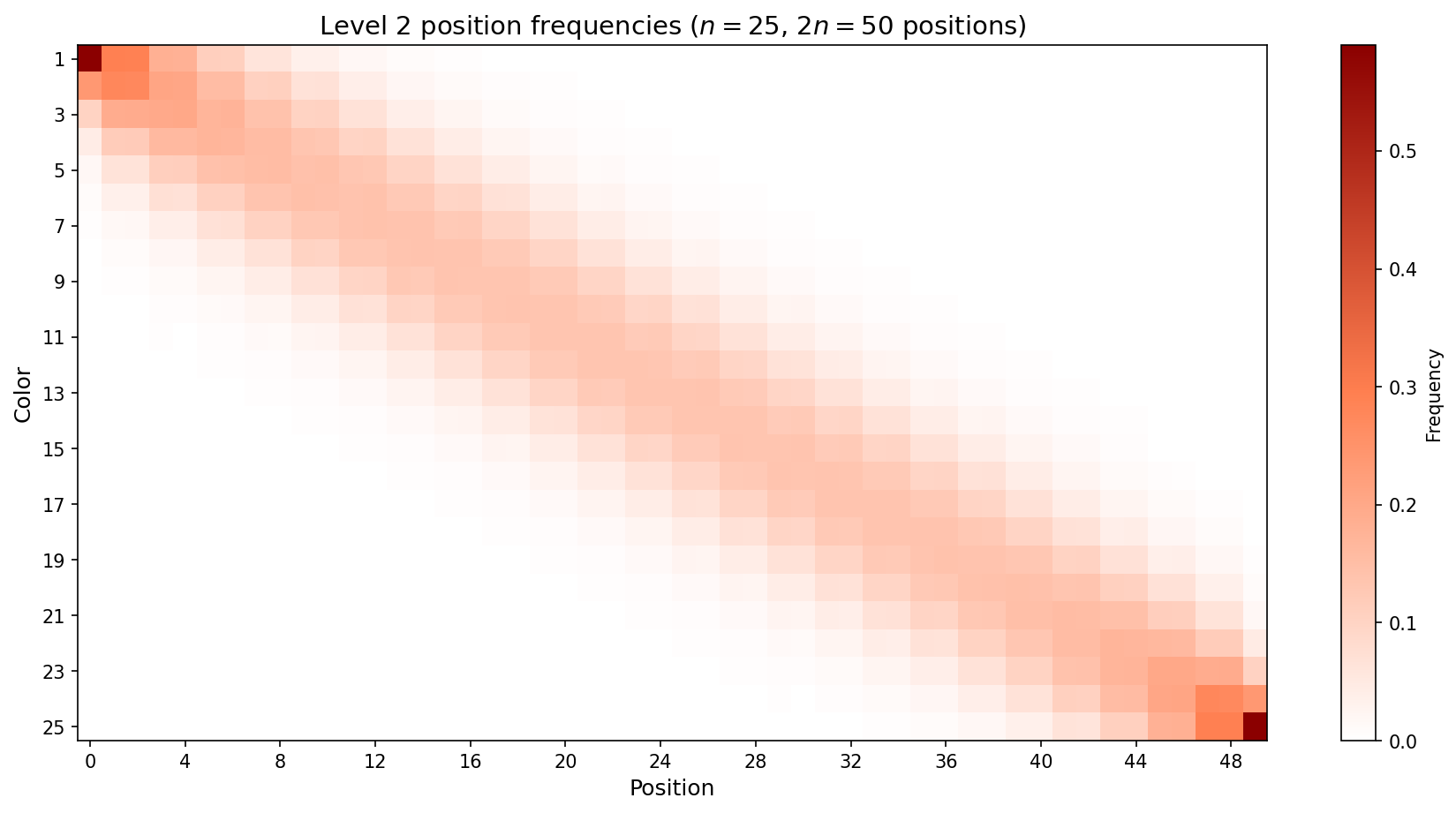}

	\vspace{0.4em}

	\includegraphics[width=0.39\textwidth]{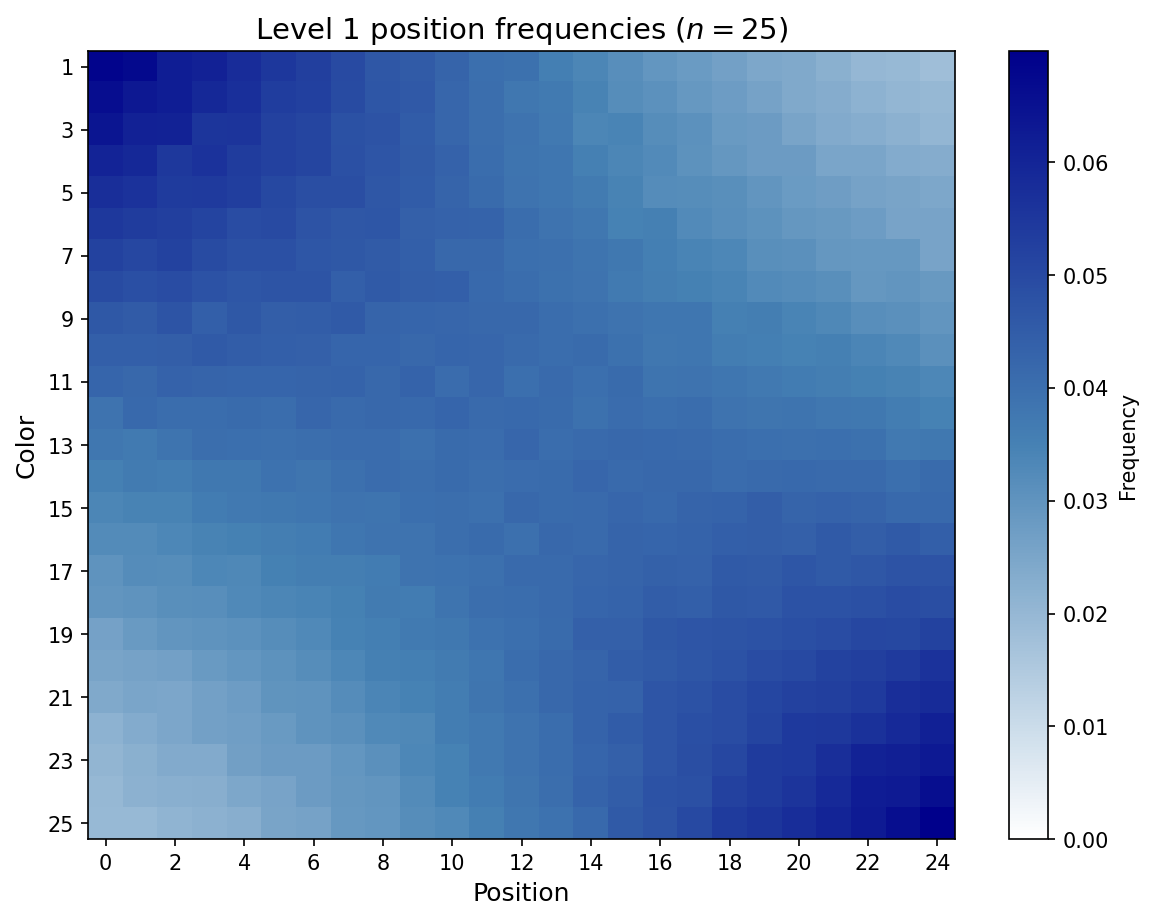}
	\hfill
	\includegraphics[width=0.57\textwidth]{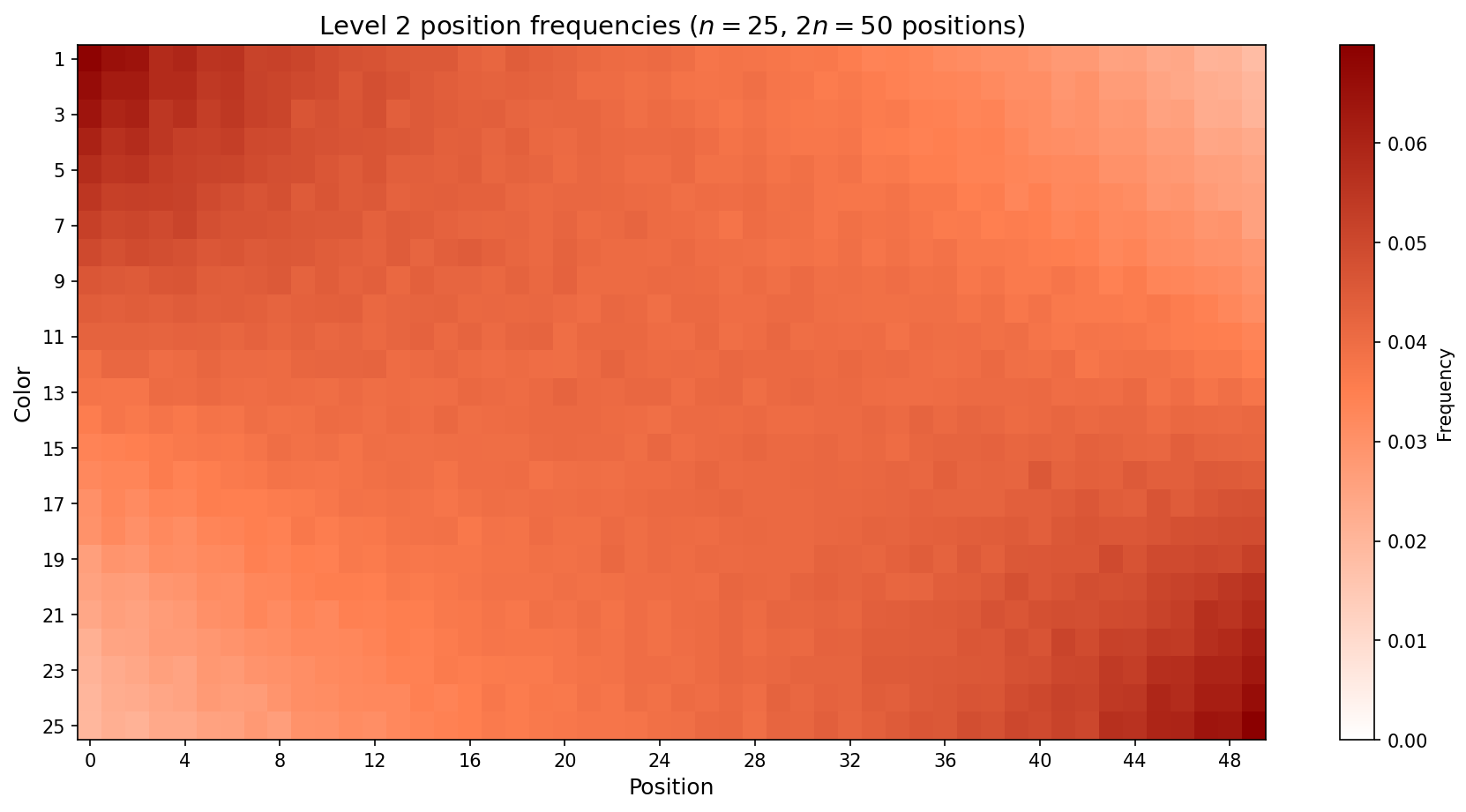}

\vspace{-2mm}

	\caption{Position frequency heatmaps from $10^5$ samples collected after
	$10^6$ ``burn-in'' steps, with $10^4$ steps between samples (so, a total of
	$10^9$ MCMC steps), where $n=25$.
	Left column: Level~$1$ ($n \times n$).
	Right column: Level~$2$ ($n \times 2n$).
	Top row: $q=0.2$.
	Bottom row: $q=0.9$.} \label{fig:sampling_heatmaps}
\end{figure}

\begin{figure}[!ht]
	\centering
	\includegraphics[width=0.7\textwidth]{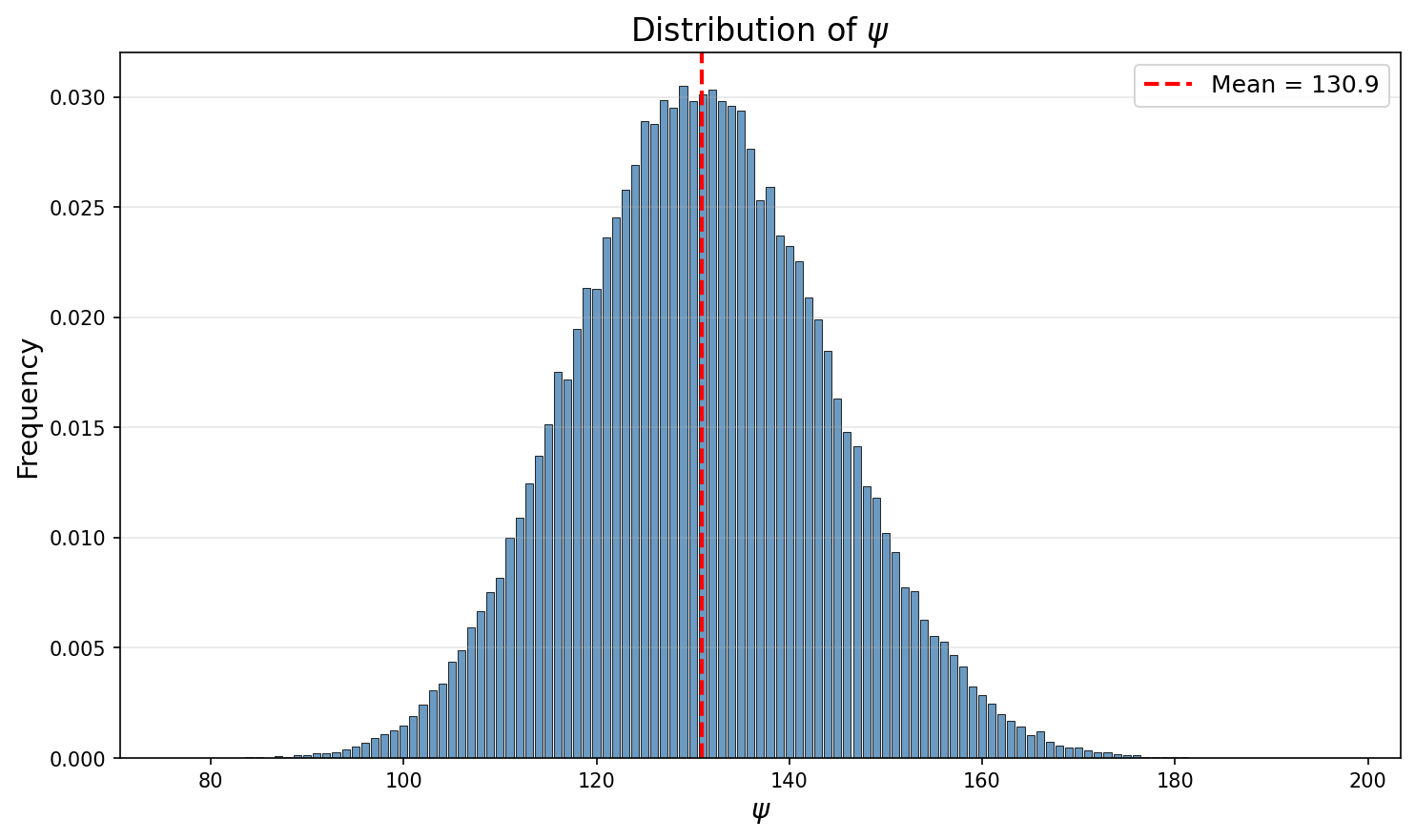}
\vspace{-2mm}

	\caption{Empirical distribution of $\psi$ with $n=25$, $q=0.9$, from the
	same samples as in Figure~\ref{fig:sampling_heatmaps}.} \label{fig:psi_histogram}
\end{figure}

\subsection*{Acknowledgements}

NB was partially supported by the EPSRC grant EP/V048902.
LP was partially supported by the NSF grant DMS-2153869 and by the Simons
Foundation Travel Support for Mathematicians Awards.
Part of this research was performed in Spring 2024 while LP was visiting the
program ``Geometry, Statistical Mechanics, and Integrability'' at the Institute
for Pure and Applied Mathematics (IPAM), which is supported by the NSF grant
DMS-1925919.
Part of this research was performed in July 2026
while both authors were attending the 4th Simons Math Summer
Workshop ``Algebraic methods in probability'' at the Simons Center
for Geometry and Physics (SCGP), Stony Brook University.
We thank Amol Aggarwal and Alexei Borodin for helpful discussions.
We are grateful to Bishal Deb, whose encyclopedic knowledge of Genocchi medians brought to our attention additional connections with the combinatorics literature. Pursuing these connections led to substantial improvements in the paper.

\pdfbookmark[1]{References}{ref}
 \LastPageEnding

\end{document}